\documentclass[pdflatex]{sn-jnl}

\usepackage{amsmath,amssymb,amsfonts,amsthm}
\usepackage{graphicx}%
\usepackage{multirow}%
\usepackage{lmodern}
\usepackage[title]{appendix}%
\usepackage{xcolor}%
\usepackage{textcomp}%
\usepackage{manyfoot}%
\usepackage{booktabs}%
\usepackage{listings}%
\usepackage{upgreek}
\usepackage{xspace}
\usepackage{mathtools}%
\usepackage{enumerate}%
\usepackage[english]{babel}

\newcommand{\eps}{\ensuremath{\varepsilon}}



\newcommand{\Yildirim}{Y{\i}ld{\i}r{\i}m\xspace}
\newcommand{\Ozluk}{\"{O}zl\"{u}k\xspace}
\newcommand{\dx}{\mathrm{d}}

\definecolor{carmine}{rgb}{0.7, 0.1, 0.09}

\definecolor{debianred}{rgb}{0.84, 0.04, 0.33}

\newtheoremstyle{sltheorems}
{10pt}
{6pt}
{\sl}
{}
{\bfseries}
{.}
{.5em}
{\thmname{#1}\thmnumber{~#2}\thmnote{~(#3)}}

\theoremstyle{sltheorems} 
\newtheorem{Thm}{Theorem}
\newtheorem*{Thm*}{Theorem}
\newtheorem{CThm}{Classical Theorem} 
\newtheorem{Cor}[Thm]{Corollary}
\newtheorem{conj}{Conjecture}

\newtheorem{lem}{Lemma}

\newtheorem*{cor*}{Corollary}

\newtheorem{Rem}{Remark}

\theoremstyle{definition}

\begin{document}

\title[Pair Correlation of zeros of Dirichlet $L$-Functions]
{Pair Correlation of zeros of Dirichlet $L$-Functions: \\
A possible path towards the conjectures of Chowla, Elliott--Halberstam and Montgomery}


\author*[1]{\fnm{Neelam} \sur{Kandhil}}\email{\small{neelamkandhil091@gmail.com}}
\equalcont{\small{These authors contributed equally to this work.}}
\author[2]{\fnm{Alessandro} \sur{Languasco}}\email{\small{alessandro.languasco@unipd.it}}
\equalcont{\small{These authors contributed equally to this work.}}
\author[3]{\fnm{Pieter} \sur{Moree}}\email{\small{moree@mpim-bonn.mpg.de}}
\equalcont{\small{These authors contributed equally to this work.}}

\affil[1,3]{\small{\orgdiv{Max-Planck-Institut f\"ur Mathematik}, 
\orgaddress{\street{Vivatsgasse 7}, \city{Bonn}, \postcode{D-53111}, \country{Germany}}}}

\affil[2]{\small{\orgdiv{Dipartimento di Ingegneria dell'Informazione - DEI}, \orgname{Universit\`a di Padova},
\orgaddress{\street{Via Gradenigo 6/b}, \city{Padova}, \postcode{35131}, \country{Italy}}}}

\abstract{Assuming the Generalized Riemann Hypothesis and a pair correlation conjecture 
for the zeros of Dirichlet $L$-functions, we establish the truth of  a conjecture
of Montgomery (in its corrected form stated by Friedlander and Granville) on the magnitude
 of the error term in  the prime number theorem in arithmetic progressions.
As a consequence, we obtain  that, under the same assumptions, the Elliott--Halberstam 
conjecture holds true. As another consequence, under the same assumptions, we will show 
that the number of  Dirichlet characters $\chi \pmod{q}$ for which $L(\frac{1}{2},\chi)=0$ 
is of order less than  $q^{1/2+\eps}$.}



\maketitle

\vskip-0.75cm
\section{Introduction}
The study of pair correlation of the zeros of the Riemann zeta function and Dirichlet 
$L$-functions has its origin in the early 20th century.
In the 1930s, Bohr and others  (Landau, Hardy)
investigated the distribution of zeros in the critical strip.
In 1942, Selberg \cite{selberg}  proved that a 
positive density of the zeros of Riemann zeta function are of odd 
multiplicity and lie on the critical line. 
The density he obtained is around 
$10^{-8}$, which was dramatically 
increased in 1974 by Levinson \cite{Levinson1974}, who improved it to more than one third.
Currently it is known,  due to
Bui, Conrey and Young \cite{BCY2011}, that more than $40\%$ of the zeros are simple and 
on the critical line.

In 1973, Montgomery made a major breakthrough
and conjectured that the pair correlation of zeros of the zeta function follows a 
distribution similar to that of the eigenvalues of 
random complex Hermitian or unitary matrices of large orders.
That there might be such a connection, was an idea that first arose during a discussion he had
with the physicist Dyson.
In the intervening years, much of the focus shifted to formalizing this connection and 
examining its implications.
A way to approach the pair correlation of Riemann zeros is to
try to asymptotically evaluate sums of the form 
\[
\sum_{\substack{0 < \gamma_j \leq T, \, j = 1,2 \\ 
\zeta(1/2+ i\gamma_j) = 0}} f(\gamma_1 - \gamma_2),
\]
with $f$ taken from a class of functions as large as possible. Using Fourier analysis, see,
e.g.~\cite{Montgomery1973}, it can be shown that 
\begin{equation}
\label{fouriertransform}
\sum_{\substack{
0 < \gamma_j \leq T, \, j = 1,2 \\ 
\zeta(1/2 + i\gamma_j) = 0}} 
r\Bigl(\frac{\gamma_1 - \gamma_2}{2\pi}\Bigr) \, W(\gamma_1 - \gamma_2) 
= 
\int_{-\infty}^{\infty} F(e^x,T)\, \widehat{r}(x) \, \dx x,
\end{equation}
where 
\[
W(u) \coloneq 
\frac{4}{4+u^2},
\quad 
\widehat{r}(x) 
\coloneq
\int_{-\infty}^{\infty} r(u) e^{-2\pi ixu} \, \dx u,
\]
is the Fourier transform of $r \in L^1(\mathbb{R})$,
and 
\[
F(x,T) 
\coloneq
\sum_{\substack{ 0 < \gamma_j \leq T, \, j = 1,2 \\  \zeta(1/2 + i\gamma_j) = 0}} 
x^{i( \gamma_1 - \gamma_2)}\, 
W(\gamma_1 - \gamma_2).
\]
In 1973, Montgomery \cite{Montgomery1973} proved, assuming the 
Riemann Hypothesis (RH), that
\begin{equation}
\label{PC-proved}
F(x, T) \sim  
\frac{T}{2\pi}\log x,
\quad \textrm{uniformly for} \ 1\leq x \leq T^{1-\eps},\,
T \to \infty,
\end{equation}
a range he extended in 1987, together with Goldston
\cite{GM1984} to $1\leq  x\leq  T$.
In \cite{Montgomery1973}, Montgomery
also formulated the  \emph{pair correlation conjecture},
stating that
\begin{equation}
\label{PC-conjecture}
F(x, T)  \sim 
\frac{T}{2\pi}\log T, 
\quad \textrm{uniformly for} \ T\leq  x \leq  T^A
\ \textrm{for every} \ A\ge 1
\end{equation}
(for the heuristic arguments that led him to believe in this,  see \cite[Sect.~3]{Montgomery1973}).

For fixed $\alpha \leq \beta$ it can be seen 
using \eqref{fouriertransform} that this conjecture is equivalent to

\begin{equation}
\label{dyson}
\sum_{\substack{0<\gamma_1, \gamma_2 \leq  T\\ \frac{2\pi \alpha}{\log T} 
\leq \gamma_1 - \gamma_2 \leq \frac{2\pi \beta}{\log T} }}
1
\sim
\Bigl(
\int_{\alpha}^{\beta} 
\Bigl( 
1 - \Bigl( \frac{\sin(\pi u)}{\pi u} \Bigr)^2  \, \dx u \Bigr)
+\delta(\alpha, \beta)
\Bigr)
\frac{T}{2\pi} \log T,
\end{equation}
as $T$ tends to infinity,
where $\gamma_1, \gamma_2$ are the imaginary parts of the zeros of the Riemann 
zeta function on the critical line and $\delta(\alpha, \beta) = 1$ if $0 \in [\alpha, \beta]$, and $0$ otherwise. 
For an accessible introduction to the pair correlation conjecture, see, e.g.,~Goldston \cite{PCintro}.

In the late 20th century, Odlyzko \cite{Ody87,Odlyzko2001} numerically studied the 
distribution of spacings between zeros of the Riemann zeta function. He
computed $10$ billion zeros near $10^{22}$-nd zero of Riemann zeta and verified the 
Riemann Hypothesis for those zeros.
Additionally, he found that the spaces between these zeros are closely distributed 
according to \eqref{dyson}.
Assuming RH, he showed that the estimate \eqref{PC-proved}
implies that at least $2/3$ of the zeros of Riemann zeta function are simple
(it is expected that all the zeros are simple). 
This is a much stronger result than the unconditional results mentioned 
in the first paragraph of this article.
In 1982, \Ozluk \cite{ozlukthesis} in his thesis, see
also \cite{OzlukBanff,ozlukjnt}, studied the 
Dirichlet $L$-function analogue of Montgomery's conjecture.
Consequently, under the Riemann Hypothesis for Dirichlet $L$-functions,
also known as the Generalized Riemann Hypothesis (GRH),   he
showed in \cite{ozlukjnt} that at least $11/12$ of all the zeros of 
such functions are simple
(allowing for a certain weight function).

In 1991, \Yildirim \cite{yal} studied the pair correlation of zeros of Dirichlet $L$-functions 
and an analogue of the conjecture stated in equation
\eqref{PC-conjecture}. In the sequel, we will work with a redefined pair correlation function. 
The key difference is that our summation will not be restricted to the zeros in the upper half-plane, 
but will also include those in the lower half-plane (note that the imaginary parts of the zeros of 
Dirichlet $L$-functions are typically not 
symmetrically distributed with respect to the real axis). In order to describe this in more detail, 
we need some notations.

From now on, we assume the 
Riemann Hypothesis for Dirichlet $L$-functions (GRH).
Let $q$ be a natural number,
$\chi$ a Dirichlet character modulo $q$ and $L(s, \chi)$  the associated Dirichlet $L$-function. 
Let $\gamma_j$ be the imaginary part of the $j$th zero (ordered by height on the half line).
Given two Dirichlet characters $\chi_1$ and $\chi_2$ modulo $q$, one can wonder to what 
extent the zeros of $L(s,\chi_1)$ are correlated with those of $L(s,\chi_2)$.
In order to measure this, we define
\[
G_{\chi_1,\chi_2}(x,T) 
\coloneq
\sum_{\substack{ \vert \gamma_j \vert \leq T, \, j = 1,2 \\ L(1/2+ i\gamma_j, \chi_j) = 0}} 
x^{i(\gamma_1 - \gamma_2)} W(\gamma_1 - \gamma_2).
\]

Henceforth, throughout the article, whenever we write $\gamma$ and $\gamma_j$ in the summation 
without additional specifications, we assume, respectively, that $L(1/2 + i\gamma, \chi) = 0$ 
and $L(1/2 + i\gamma_j, \chi_j) = 0$. 
Moreover, we remark that 
rather than $0<\gamma\leq  T$, we require that $ \vert \gamma \vert \leq  T$,
thus taking into account that the imaginary parts 
of the zeros of Dirichlet $L$-functions are (mostly) not symmetrical with respect to the real axis.
Since it currently cannot be excluded that $L(\frac{1}{2},\chi)=0$ for
many characters $\chi$, we are forced to allow for $\gamma_j=0$ in the definition of $G_{\chi_1,\chi_2}(x,T)$ too.
It is expected that $L(\frac{1}{2}, \chi) \ne 0$ for all primitive characters $\chi$, but this remains unproved.
The first to conjecture this was Chowla \cite{chowla}, in the case of  quadratic characters.
Soundararajan \cite[p.~454]{Sound} proved that the proportion of Dirichlet $L$-functions $L(s, \chi_{8d})$,
with $d$ odd and squarefree, which do not vanish at $s=1/2$ is at least $7/8$.
\Ozluk and Snyder \cite{OzlukSnyder}
showed, assuming GRH, that $L(\frac{1}{2}, \chi_d) \ne 0$ for at 
least $15/16$ of the fundamental discriminants $ \vert d \vert  \leq  X$.
Khan, Mili\'cevi\'c and Ngo \cite{KMN}
proved that $L(\frac{1}{2},\chi) \ne 0$ 
for at least $5/13-\eps$ of all 
primitive Dirichlet characters $\chi$ of modulus $q$, 
for all primes $q$ large enough in terms of $\eps>0$. 
Qin and Wu \cite{QinWu2025}
proved that
for at least $7/19$ of the primitive Dirichlet characters $\chi$ with 
large general modulus, the central value $L(\frac{1}{2},\chi)$ is non-vanishing.
Very recently, Leung \cite{Leung2025} showed that at least
$35.9\%$ is attainable, if one further restricts to $q$ large, square-free and smooth.
Matom\"aki and {\v C}ech \cite[Cor.~1]{marmat} showed that
if $\psi$ is a real primitive character modulo $D$ with $L(1,\psi) \ll (\log D)^{-25 -\eps}$, 
then for any prime $q \in [D^{300}, D^{O(1)}]$, 
one has $L(\frac{1}{2}, \chi) \ne 0$ for almost all $\chi$ mod $q$.
The interested reader may consult \cite{arm,NK,serre}
for further details.

Note that $G_{\chi_1,\chi_2}(x,T)$ is an analogue of $F(x,T)$.
We make a more global object out of 
$G_{\chi_1,\chi_2}(x,T)$
by considering for every $a$ coprime with $q$ the quantity
\begin{equation}
\label{Fq-def}
F_q(x,T) 
=
F_q(x,T;a)
\coloneq
\sum_{\chi_1,\chi_2 \pmod{q}} \overline{\chi_1}(a) \chi_2(a)G_{\chi_1, \chi_2} (x,T).
\end{equation}

Observe that  $F_q(x,T)$  depends on $a$. However, in our estimates the actual value of $a$ 
will not play a role and so, in accordance with the existent literature, we suppress the $a$ dependence.
It is worth remarking that $F_q(x,T)$ is real and non-negative.
To prove this we need to introduce some ideas, which we will
exploit also in Section \ref{S4}, contained in 
a paper by Heath-Brown \cite{HB}.
We need first to remark that an easy computation gives
$W(u) = \int_{-\infty}^{\infty} e^{ivu}e^{-2 \vert v \vert }\, \dx v$ and
that, letting 
\begin{equation}
\label{Sigma-def}
\Sigma(x,T,v) 
=
\Sigma(x,T,v;q,a) 
\coloneq
\sum_{\chi \pmod{q}}
\overline{\chi} (a)
\sum_{ \vert \gamma \vert \leq T} x^{i\gamma}e^{iv\gamma},
\end{equation}
another easy computation shows that
\begin{equation}
\label{integral-form-of-F}
F_q(x,T) = 
\int_{-\infty}^{\infty} \vert \Sigma(x,T,v) \vert ^2\,e^{-2 \vert v \vert }\, \dx v,
\end{equation}
implying that $F_q(x,T)\ge 0$.
%
Now, using  \cite[Ch.~16, eq.~(1)]{Davenport}  and  \cite[Ch.~16, Lemma]{Davenport}, 
we have 
\begin{equation*}
    \vert G_{\chi_1, \chi_2}(x,T) \vert  
   \le
   \sum_{\substack{ \vert \gamma_j \vert  \leq T, \, j = 1, 2 
   }}
\frac{1}{4 + (\gamma_1 - \gamma_2)^2}
\ll 
T \log^2 (qT),
\end{equation*}
uniformly in $x$ as $T \to \infty$.
Therefore, we trivially have
\begin{equation}
\label{trivial}
   F_q(x,T)
\ll T (\varphi(q) \log (qT))^2,
\end{equation}
uniformly in $x$ as $T \to \infty$,
where $\varphi$ denotes the Euler's totient function.
\Yildirim \cite{yal} defined the 
$G_{\chi_1, \chi_2}(x,T)$ and $F_q(x,T)$ analogues 

\[
G^{+}_{\chi_1, \chi_2}(x,T) 
\coloneq 
\sum_{\substack{0 \leq  \gamma_j \leq T \\
}} x^{i(\gamma_1 - \gamma_2)} W(\gamma_1 - \gamma_2)
\]
and
\begin{equation}
\label{Fq-plus-def}
F^{+}_q(x,T)
=
F^{+}_q(x,T; a)
\coloneq
\sum_{\chi_1,\chi_2 \pmod{q}}
\overline{\chi_1}(a) \chi_2(a)
G^{+}_{\chi_1, \chi_2}(x,T),
\end{equation}
and  proved the following theorem\footnote{In 
fact, \Yildirim defined $G^{+}_{\chi_1, \chi_2}(x,T)$ using positive values
of $\gamma_j$, $j=1,2$, only. But inspecting his proof,
which is essentially Montgomery's original one, 
it is easy to see that it works with $0\le\gamma_j\leq  T$, $j=1,2$,
instead (see in particular the displayed equation after (25)
on page 187 of \cite{Montgomery1973}).}.

\begin{Thm}[\Yildirim \cite{yal}]
\label{yalthmthm}
Under GRH, as \( x \to \infty \), 
we have
\begin{equation}
\label{with2}
F^{+}_q(x,T)
\sim  \frac{\varphi(q)}{2\pi} T\log x
\end{equation}
uniformly for every $a$ coprime with $q$ and
\[
1 \leq q \leq  \sqrt{x}/ \log^{3}x 
\quad
\textrm{and}
\quad
\frac{x}{q}\log x \leq T \leq \exp(x^{1/4}).
\]
\end{Thm}

For smaller $T$ (with respect to $x$), he conjectured:
\begin{conj}[Pair correlation for Dirichlet 
$L$-functions \cite{yal}]
\label{Yild_conj}
Let $q=1$ or $q$ be a prime 
and \( 0 < \eta \leq 1 \) be a fixed real
number. 
Then, under GRH, as \( x \to \infty \),
\begin{equation}
\label{turkey}
F^{+}_q(x,T)
\sim  \frac{\varphi(q)}{2\pi} T\log (qT),
\end{equation}
uniformly for every $a$ coprime with $q$ and
\[
q \leq  \min(\sqrt{x}/ \log^{3}x, x^{1-\eta} \log x)
\quad
\textrm{and}
\quad
x^{\eta} \leq T < \frac{x}{q} \log x.
\]
\end{conj}

\Yildirim in Conjecture \ref{Yild_conj} and in his results about the 
mean square assumed that $q$ is prime in order
to avoid the contribution of imprimitive characters; 
a restriction we do not have
in our application.
The difference between $G_{\chi_1, \chi_2} (x,T)$
and $G^+_{\chi_1, \chi_2} (x,T)$, and hence between $F_q(x,T)$
with $F^+_q(x,T)$, is that in the 
latter quantity the sum is 
performed only on the 
non-negative imaginary parts.
We, however, are forced to work
instead with the
condition $\vert \gamma \vert \leq  T$, since
the imaginary parts  of the zeros of Dirichlet $L$-functions are 
(mostly) not symmetrical with respect to the real axis
and it is possible that $L(\frac{1}{2},\chi)$ vanishes.
This is in fact standard, see, for example, also the explicit formula 
lemmas, Lemmas \ref{davenlem2}-\ref{davenlem}, and
\eqref{conj-trick}-\eqref{explform-positive-gammas}.
We also remark that we are able to relate $F_q(x,T)$
with $F^+_q(x,T)$ only if we assume that
there are no Dirichlet $L$-functions having
a zero attained at $1/2$.  
To show this, we need some definition first.
Let $a$ coprime with $q$ be fixed, put
\begin{align}
\label{Sigma-plus-def}
\Sigma^+(x,T,v) 
&=
\Sigma^+(x,T,v;q,a) 
\coloneq
\sum_{\chi \pmod{q}}
\overline{\chi} (a)
\sum_{0 \leq  \gamma \leq T} x^{i\gamma}e^{iv\gamma},
\\
\notag
\Sigma^-(x,T,v) 
&=
\Sigma^-(x,T,v;q,a) 
\coloneq
\sum_{\chi \pmod{q}}
\overline{\chi} (a)
\sum_{-T \leq  \gamma \leq 0} x^{i\gamma}e^{iv\gamma},
\end{align}
and
\[
\Sigma_0 \coloneq \sum_{\substack{\chi \pmod{q} \\ L(1/2,\chi)=0}} \overline{\chi} (a).
\]
Using \eqref{Sigma-def}, we have
\[
\Sigma(x,T,v)= \Sigma^+(x,T,v)+\Sigma^-(x,T,v) -  \Sigma_0
\]
and, recalling that $L(1/2 + i\gamma, \chi) = 0$ if and only if 
$L(1/2 - i\gamma, \overline{\chi}) = 0$,  we  obtain
\begin{equation}
\label{Sigma-conj}
\Sigma^-(x,T,v) = \overline{\Sigma^+(x,T,v)},
\end{equation}
immediately leading to
\begin{equation}
\label{Sigma-ineq}
\vert \Sigma(x,T,v) \vert 
\leq  \vert \Sigma^+(x,T,v) \vert + \vert \Sigma^-(x,T,v) \vert  
+ \vert \Sigma_0 \vert 
= 2\vert \Sigma^+(x,T,v) \vert + \vert \Sigma_0 \vert .
\end{equation}
Recalling now \eqref{Fq-plus-def}, we easily obtain, cf.\,\eqref{integral-form-of-F}, that
\begin{equation}
\label{integral-form-of-F-plus}
F^+_q(x,T) 
= 
\int_{-\infty}^{\infty} \vert \Sigma^+(x,T,v) \vert ^2\,e^{-2 \vert v \vert }\, \dx v
\ge 0
\end{equation}
and, using \eqref{Sigma-conj}, also that
\begin{align*}
F^{-}_q(x,T)&= F^{-}_q(x,T;a) \coloneq\sum_{\chi_1,\chi_2 \pmod{q}}
\overline{\chi_1}(a) \chi_2(a)
\sum_{\substack{-T \leq  \gamma_j \leq 0 \\  j = 1,2 
}} x^{i(\gamma_1 - \gamma_2)} W(\gamma_1 - \gamma_2)
\\
&= \int_{-\infty}^{\infty} \vert \Sigma^-(x,T,v) \vert ^2\,e^{-2 \vert v \vert }\, \dx v
=F^+_q(x,T) .
\end{align*}
Finally, using \eqref{integral-form-of-F}, \eqref{Sigma-ineq}-\eqref{integral-form-of-F-plus}, 
the Cauchy--Schwarz inequality \eqref{Cauchy} and $\int_{-\infty}^{\infty}\,e^{-2 \vert v \vert }\, \dx v =1$, 
we obtain
\begin{align*}
F_q(x,T) 
&
\le
4\int_{-\infty}^{\infty} \vert \Sigma^+(x,T,v) \vert ^2\,e^{-2 \vert v \vert }\, \dx v 
+
\vert \Sigma_0 \vert^2 + 4 \vert \Sigma_0 \vert \int_{-\infty}^{\infty} \vert \Sigma^+(x,T,v) \vert \,e^{-2 \vert v \vert }\, \dx v 
\\&
\leq  4 F^+_q(x,T)  + \vert \Sigma_0 \vert^2 + 
4 \vert \Sigma_0  \vert  (F^+_q(x,T))^{1/2}.
\end{align*}
Assuming that there are no Dirichlet $L$-functions that vanish at $1/2$
(so that $\Sigma_0=0$),
we then have $F_q(x,T) \leq  4 F^+_q(x,T)$ and $F_q(x,T) \leq  4 F^{-}_q(x,T)$.
In general, though, we cannot show that relations of this
kind hold true and hence we  are forced
to prove distinct results and 
to formulate conjectural upper bounds for 
both $F_q(x,T)$ and $F^+_q(x,T)$ functions separately.
\begin{Thm}
\label{yalmodifiedcor}
Let  $\eps>0$ be arbitrary and fixed.
Under GRH, as \( x \to \infty \),  we have
\[
F_q(x,T)  
\ll \varphi(q)T \log x
\quad
\textrm{and}
\quad
F^{+}_q(x,T) 
\ll \varphi(q)T \log x
\]
uniformly for  every $a$ coprime with $q$ and
\[
1 \leq q \leq  x^{1-\eps}
\quad\textrm{and} \quad
x \leq T
\leq \exp(\sqrt{x}).
\]
\end{Thm}

\begin{Rem}
\label{remark1}
In fact, it is easy to see that 
the $q$-uniformity 
range in Theorem \ref{yalmodifiedcor}
can be slightly extended to $1 \leq q \leq  x/ \log^{(1+\eps)} x$ by using the Brun--Titchmarsh estimate
in its whole range of validity.
The price paid for this is that
only the slightly weaker estimates  
$F_q(x,T) \ll\varphi(q) T  \log^2 x/\log \log x$ and 
$F^{+}_q(x,T)  \ll\varphi(q) T  \log^2 x/\log \log x$ are obtained. 
\end{Rem}

The different uniformity ranges for $q,T$ in our result compared with the ones 
in \Yildirim's are  due to the fact we wish to work 
with $q$ larger than $\sqrt{x}$, so outside the range in which the 
Bombieri--Vinogradov theorem can be applied. As a consequence, having $q$ so large forces us
to have $T$ larger than $x$ (rather than 
$\frac{x}{q}\log x$, as in \Yildirim's result).
As a penalty we only obtain an upper bound for $F_q(x,T)$, rather than an asymptotic formula. 
However, this suffices for our purposes.

Motivated by conjectures made by Montgomery, see \eqref{PC-conjecture}, and \Yildirim,
see \eqref{turkey}, for the remaining values of $T$ with respect to $x$, we 
believe:
\begin{conj}
\label{yalconj1}
Let $\eps>0$ be arbitrary and fixed. Under GRH, as \( x \to \infty \), 
we have
\begin{equation}
\label{turkey2}
F_q(x,T) \ll
\varphi(q)T x^{ \eps}
\quad
\textrm{and}
\quad
F^{+}_q(x,T) \ll
\varphi(q)T x^{ \eps},
\end{equation}
uniformly for  every $a$ coprime with $q$, 
with
\[
1 \leq q \leq  x^{1-\eps} \quad\textrm{and}\quad  
x^{\eps} \leq  T < x. 
\]
\end{conj}
One could wonder why in  Conjecture \ref{yalconj1} we wrote the bounds for 
$F_q(x,T), F^+_q(x,T)$ with $x^\eps$  instead of with $\log^2 x$ 
(as Theorem \ref{yalmodifiedcor} and Remark \ref{remark1} would suggest)
and we limited the $q$-uniformity range  to $q\leq  x^{1-\eps}$. 
Both choices are due to the fact that 
with more daring assumptions we would obtain results 
that are in contradiction with the ones established by Friedlander--Granville  \cite[pp.~365--366]{FriedlanderG1989}
on the Montgomery and Elliott--Halberstam conjectures.

In Theorem \ref{yalmodifiedthm1} we extend the
allowed $T$-range of Theorem \ref{yalthmthm},
as we show that the logarithmic term from the lower 
bound for $T$ can be removed by utilizing the Goldston--Montgomery version
of Montgomery--Vaughan's mean value theorem for Dirichlet series 
(see Lemma \ref{gmslem}), instead of the 
original version given in \cite{Montvaug}. Furthermore, 
correcting a typo in \cite{yal} 
(see Footnote \ref{foot_typo} below on page \pageref{foot_typo}), 
we are able
to replace $\exp(x^{1/4})$ by $\exp(x^{3/4})$ in the upper bound  for $T$.
\begin{Thm}\label{yalmodifiedthm1}
Let $\eps>0$ be arbitrary and fixed.
Under GRH, as \( x \to \infty \), we have
\begin{equation}
\label{no2}
F_q(x,T)  
\sim \frac{\varphi(q)}{\pi} T\log x
\end{equation}
uniformly for  every $a$ coprime with $q$ and
\[
1 \leq q \leq  \sqrt{x}/\log^{2+\eps}x,
\quad\textrm{and} \quad
\frac{x}{\varphi(q)}   \leq T \leq 
\exp(x^{3/4}).
\]
\end{Thm}

Comparison of \eqref{with2} and \eqref{no2} shows that $F_q(x,T)$ grows twice as fast as $F_q^+(x,T)$.
This is due to the fact that we are summing the imaginary parts of the zeros over a range 
which is twice the one used in \eqref{with2}.

Note that in the $q$-aspect the right hand side
in \eqref{no2} is the square root of the trivial estimate \eqref{trivial}.
The orthogonality of the Dirichlet characters plays a key role in establishing this.

The consequences on the simplicity of the zeros mentioned in the first paragraph of this article aside,
it is not surprising that these results and conjectures can give information
on the distribution of primes too.
In 1978, Gallagher and Mueller \cite{GM}
were the first to establish a result
connecting the pair correlation of zeros of $\zeta(s)$ and prime number distribution.
They proved that assuming RH, and pair correlation
conjecture for zeros of Riemann zeta function,
\begin{equation}
\label{psilittleo}    
\psi(x) = x + o(\sqrt{x} \log^2 x)
\quad \textrm{as}\ x \to \infty.\footnote{We recall some standard prime number theory notations in Section \ref{S2}.}
\end{equation}
This improves on the classic result 
that, assuming RH, the above estimate holds with error term $O(\sqrt{x} \log^2 x)$.
In 1982, Heath-Brown \cite{HB} was able to obtain \eqref{psilittleo} 
assuming the weaker hypothesis $F(x, T)= o( T\log^2 T),$ as $T \to \infty$, uniformly for 
$ T\leq   x \leq  T^A$ for every constant $A \geq 1$.
In the same year, Goldston and Montgomery \cite[Thm.~2]{GM1984} established an equivalence 
between an asymptotic result for the distribution of primes and the pair correlation conjecture \eqref{PC-conjecture}. 
\Yildirim \cite{yal}, assuming Conjecture \ref{Yild_conj}, was able to 
establish, following the work of Goldston--Montgomery \cite{GM1984}, 
an asymptotic formula for the mean square of primes in arithmetic progressions.\footnote{However, it 
seems that \Yildirim did not take into account that a Dirichlet $L$-function may have a zero at $1/2$.}

We recall that in 1837, Dirichlet (see for
example \cite[Chp.~2]{nar}) proved that there are infinitely many primes congruent to
$a \pmod{q}$ for $(a,q)=1$. 
The method used by Hadamard to prove
the Prime Number Theorem
allowed him to show also, cf.~\cite[p.~206]{nar}, that
\begin{equation}
\label{dir}
 \psi(x;q,a) - \frac{x}{\varphi(q)} = o\Bigl(\frac{x}{q} \Bigr),
\end{equation}
as $x \to \infty.$ 
In 1936, Walfisz proved, for any fixed $A>0$, that 
\eqref{dir} holds uniformly for $q\leq  \log^A x$.
Under GRH,
for any $(a, q)= 1$ it is known, uniformly for
$1 \leq q < x $, that 
\begin{equation}
\label{gr1}
 \psi(x;q,a) - \frac{x}{\varphi(q)} = O\bigl( \sqrt{x}\, \log^2 x \bigr)
\end{equation}
as $x \to \infty.$ See for instance, \cite[p.~125]{Davenport}.
Moreover, estimate \eqref{gr1} implies that under GRH,
for any $\eps >0, (a, q)= 1$, the estimate \eqref{dir} holds
uniformly for $1 \leq q \leq \sqrt{x}/ \log^{2+\eps}x$,
as $x \to \infty.$

Friedlander and Granville \cite[Cor.~2]{FriedlanderG1989} proved that
the estimate \eqref{dir} cannot hold uniformly for the range
$1 \leq q < x/ \log^B x$, where $B>0$ is any arbitrary fixed number. 
In the same paper they also formulated the following corrected
form of another conjecture of Montgomery,
see \cite[p.~366]{FriedlanderG1989}.

\begin{conj}[Montgomery]
\label{Mont-primes-conj}
For any $\eps >0, (a, q)= 1$, uniformly for $1\leq q < x$ we have
\begin{equation}
\label{MFG-conj}
 \psi(x;q,a) - \frac{x}{\varphi(q)} \ll \sqrt{\frac{x}{q}}\, x^{\eps}. 
\end{equation}
\end{conj}

Our main result states that if both GRH and Conjecture \ref{yalconj1} hold true, 
then the estimate \eqref{MFG-conj} 
holds in almost the whole $q$-range given in 
Montgomery's Conjecture \ref{Mont-primes-conj}.
Since for $1\leq  q \leq  x^\eps$ it is clear that equation \eqref{MFG-conj} follows immediately from \eqref{gr1}, 
it suffices to restrict our analysis to the range $x^\eps<q<x$.
\begin{Thm}\label{thmfullrange}
Let $\eps>0$ be arbitrary and fixed and $(a, q)= 1$. 
Assume that both GRH and Conjecture \ref{yalconj1} 
hold true. Then, the upper bound \eqref{MFG-conj} holds
for $x$ sufficiently large and
uniformly for $1 \leq  q \leq x^{1-\eps}$.
\end{Thm}

Another very important conjecture on the distribution of primes is:

\begin{conj}[Elliott--Halberstam \cite{ElliottH1968/69}]
\label{EHconjecture}
For every $\eps>0$ and $A>0$,
\[
\sum_{q\leq  x^{1-\eps}} \max_{(a,q)=1}\Bigl \vert \,\psi (x;q,a)-{\frac {x}{\varphi (q)}}\Bigr \vert  
\ll_{A,\eps} 
\frac{x}{\log^A x}.
\]
\end{conj} 
If true, this is close to best possible, as Friedlander and Granville \cite{FriedlanderG1989}, using the
Maier matrix method \cite{Maier}, showed the conjecture to be false when $x^{1-\epsilon}$ is 
replaced by  $x/\log^{B}x$, with $B>0$ arbitrary. Indeed, Maier's matrix method can be used 
to establish limitations of a very successful probabilistic 
prime number distribution model of Cram\'er, cf.~Granville \cite{Granvillecramer}.

So far, an estimate 
such as the one in Conjecture \ref{EHconjecture}, is known only in
the interval  $q \leq \sqrt{x}/ \log^{B}x$, with $B>0$ (thanks to the Bombieri--Vinogradov Theorem).

Theorem \ref{thmfullrange} has the following straightforward 
elegant corollary supporting the Elliott--Halberstam conjecture.
\begin{Cor}
\label{EH-type-estim}
Let $\eps>0$ be arbitrary and fixed and $(a, q)= 1$. 
Assume that both GRH and Conjecture \ref{yalconj1} 
hold true. Then,
for $x$ sufficiently large, we have 
\[
\sum_{q\leq  x^{1-4\epsilon}} 
\max_{(a,q)=1}\,
\Bigl \vert \, \psi (x;q,a)-{\frac {x}{\varphi (q)}}\Bigr \vert  
\ll
x^{1-\eps}.
\]
\end{Cor}

Our next result is about estimating the number
of Dirichlet $L$-functions modulo $q$ that might vanish 
at $1/2$. Chowla conjectured that 
this never happens in case 
$\chi$ is a quadratic character.
\begin{Thm}
\label{thm-Lzero-onehalf}
Let $\eps>0$ be arbitrary and fixed and $q$ be sufficiently large. 
Assume that both GRH and  Conjecture \ref{yalconj1} 
for $a=1$ hold, then 
\begin{equation*}
\#\Bigl\{\chi \pmod{q} \colon L\Bigl(\frac{1}{2}, \chi\Bigr)=0\Bigr\}
\ll 
q^{1/2+\eps}.
\end{equation*}
\end{Thm}
Note that $a$ in Conjecture \ref{yalconj1} is general, but here it suffices to restrict to $a=1$ only.
For ease of notation we will put
\begin{equation}
\label{zero-set-def}
Z_q\coloneq 
\Bigl\{\chi \pmod{q} \colon L\Bigl(\frac{1}{2}, \chi\Bigr)=0\Bigr\}.
\end{equation}

For an illuminating discussion of the role of a possible zero of Dirichlet $L$-functions
in $s=1/2$ in Montgomery's Conjecture \ref{Mont-primes-conj},
we refer to the introduction of Fiorilli's paper \cite{Fiorilli2015}. 
Further, we point out that Theorem \ref{thm-Lzero-onehalf} implies 
Theorem 1.3 of \cite{Fiorilli2015}.

\begin{Rem}
It is also possible to show that a  weaker form of \eqref{MFG-conj} 
follows by assuming a weaker form of the pair correlation conjecture 
for the Dirichlet $L$-functions. This weaker form involves an additional 
factor $g(q)$ satisfying $1\leq  g(q)\leq  \varphi(q)$ in the upper bound,
the original Conjecture \ref{yalconj1}
arising on taking $g(q)=1$
and, essentially, the trivial estimate \eqref{trivial} on choosing $g(q)=\varphi(q)$.
To be specific, if we replace $\varphi(q)T x^{ \eps}$ in both upper bounds 
in Conjecture \ref{yalconj1} 
with $\varphi(q) g(q) T x^{ \eps}$, the following modifications hold 
with proofs similar to those in Section~\ref{S4}.
\begin{enumerate}[a)]
\item 
We can obtain a result analogue of Theorem \ref{thmfullrange} 
 in which the final estimate is replaced by
\[
 \psi(x;q,a) - \frac{x}{\varphi(q)} \ll  \sqrt{\frac{x g(q)}{q}}\ x^{\eps};
\]

\item
we can obtain a result analogue of Theorem \ref{thm-Lzero-onehalf} 
 in which the final estimate is replaced by
\[
\#\Bigl\{\chi \pmod{q} \colon L\Bigl(\frac{1}{2}, \chi\Bigr)=0\Bigr\}
\ll 
(qg(q))^{1/2+\eps},
\]
which is non-trivial only for $g(q)\ll q^{1-\eps}$;

\item
moreover, if
$g(q)= \varphi(q)^\alpha$, $0 \leq  \alpha \leq  1$,
we can obtain the following analogue of Corollary \ref{EH-type-estim}:
\[
\sum_{q\leq  Q} 
\max_{(a,q)=1}\,
\Bigl \vert \, \psi (x;q,a)-{\frac {x}{\varphi (q)}}\Bigr \vert  
\ll
x^{1-\eps}
\]
for $Q\coloneq Q(\alpha,x)\leq  x^{1/(1+\alpha)-4\eps}$.
Note that
$x^{1/2 - 4\eps}\leq  Q \leq  x^{1 - 4\eps}$ for such $\alpha$'s.
\end{enumerate}
\end{Rem}

\medskip

 The paper is organised as follows. In Section \ref{S2}, we recall existing results we need to prove 
our theorems. In Section \ref{S3}, we prove Theorems \ref{yalmodifiedcor}-\ref{yalmodifiedthm1} 
and Section \ref{S4} is devoted to proofs of the main results of this paper, i.e, 
Theorems \ref{thmfullrange} and \ref{thm-Lzero-onehalf}. 
The crux of the proof of Theorem \ref{thmfullrange} lies in the following observations. 
As noted earlier, the estimate in Conjecture \ref{yalconj1}
represents a substantial improvement compared 
to the trivial bound in both the  $q$ and $T$ aspects when discussing the pair correlation of 
zeros of Dirichlet  $L$-functions (see \eqref{trivial}, \eqref{turkey} \& \eqref{turkey2}). 
On the other hand, it is relatively less significant in the context of the correlation of 
zeros of the Riemann zeta function. This leads to the following key point: while proving 
Theorem \ref{thmfullrange}, the analysis proceeds smoothly in case of ``not too small'' zeros,
aligning with the case of modulus $1$ (see \cite[Thm.~1]{HB}). In this zero-range,
it is enough to use $F^+_q(x,T)$ since $L(1/2 + i\gamma, \chi) =0$ 
if and only if $L(1/2 - i\gamma, \overline{\chi}) =0$, see Lemma \ref{conj-trick-lemma}. 
However, when accounting 
for the contributions of low-lying zeros, in contrast to the 
modulus $1$ case, the trivial estimate for $F^+_q(x,T)$ 
becomes inadequate. 
To address this difficulty, we need to use a suitable conjectural bound for a form 
of the pair correlation function built with zeros having imaginary parts
in a symmetric interval centered at $0$, namely $F_q(x,T)$, in order to
dissect the problem into dyadic intervals and subsequently 
transition from local to global intervals.
Moreover, the necessity of working with $F_q(x,T)$ connects
our setting to the possible existence of a zero attained at $1/2$ for Dirichlet $L$-functions
and leads to the proof of Theorem \ref{thm-Lzero-onehalf}.
For more details, please refer to Section \ref{S4}.

\section{Preliminaries}\label{S2}
\subsection{Prime number distribution}
In this section, we recall the material we need on the distribution of prime numbers, using the notations
\begin{equation*}
\pi(t) \coloneq \sum_{p \leq  t} 1, \quad \quad
\pi(t;q,a) \coloneq \sum_{\substack{p \leq  t \\ p \equiv a  \pmod{q}}} 1,
\end{equation*}
and
\begin{equation*}
\psi(t) \coloneq \sum_{n \leq  t} \Lambda(n), \quad \quad
\psi(t;q,a) \coloneq \sum_{\substack{n \leq  t \\ n \equiv a  \pmod{q}}} \Lambda(n),  
\quad \quad 
\psi(t,\chi) \coloneq \sum_{n \leq t} \Lambda(n) \chi(n),
\end{equation*}
where $\Lambda$ denotes the von Mangoldt function and $\chi$ a Dirichlet character
modulo $q$.
Two equivalent forms of the Prime Number theorem are
\begin{equation}
\label{pnt}
\pi(t)\sim
\int_{2}^{t} \frac{\, \dx u}{\log u}
\quad
\textrm{and}
\quad
\psi(t)\sim t
\quad (t\to \infty).
\end{equation}
For fixed coprime integers $a$ and $q$, we have asymptotic equidistribution: 
\begin{equation}
\label{equi}
\pi(t;q,a)\sim \frac{\pi(t)}{\varphi(q)},
\quad
\textrm{and}
\quad
\psi(t;q,a)\sim \frac{\psi(t)}{\varphi(q)}
\quad (t\to \infty),
\end{equation}
with $\varphi$ Euler's totient.
Furthermore, under GRH,
for any $(a, q)= 1$ and fixed constant $A>0$, uniformly for $q\leq \sqrt{x} \log^{-A} x$ we have 
\begin{equation}
\label{pnterror}
\pi(t;q,a)= \frac{1}{\varphi(q)} \int_{2}^{t} \frac{\, \dx u}{\log u}  + O(\sqrt{t}\log t),
\end{equation}
as  $t \to \infty$.
In addition, for the principal character $\chi_0 \pmod{q}$, we have
\begin{equation}
\label{sip}
   \psi(t,\chi_0)-\psi(t) = O((\log q)( \log t))\quad (t \to \infty).
\end{equation}

An important tool we will use is the following theorem (for a proof, see, e.g.,
Montgomery--Vaughan \cite[Sect.~3--6]{MVsieve}).
\begin{CThm}[Brun--Titchmarsh theorem]
\label{BT-thm}
Let $x,y>0$ and $a,q$ be coprime positive integers.
Then, uniformly for all $y>q$, we have
\begin{equation*} 
\pi(x+y;q,a) - \pi(x;q,a) < \frac{2y}{\varphi(q) \log(y/q)}.
\end{equation*}
\end{CThm}

Starting point for our deliberations is an explicit truncated form of the 
von Mangoldt explicit formula, which we state in the classical, respectively 
Dirichlet $L$-function case.
\begin{lem}\cite[Ch.~17]{Davenport}.
\label{davenlem2}
Let $2 \leq Z \leq x$. Assuming RH, we 
have, as $ x \to \infty,$
\[
\psi(x) = x - \sum_{\substack{ \vert \gamma \vert  \leq  Z \\
\zeta(1/2 + i\gamma) = 0}} \frac{x^{1/2 + i\gamma}}{1/2 + i\gamma} +
O\Bigl(\frac{x \log^2 (xZ)}{Z} \Bigr).
\]
\end{lem}

\begin{lem}\cite[Ch.~19]{Davenport}.\label{davenlem}
If $\chi$ is a non principal 
character modulo $q$ and $2 \leq Z \leq$ $x$, then assuming GRH, we have, as $ x \to \infty,$
\begin{equation}
\label{explform-all-gammas}
\psi(x,\chi) = - \sum_{\substack{ \vert \gamma \vert  \leq  Z 
}} \frac{x^{1/2 + i\gamma}}{1/2 + i\gamma} +
O\Bigl(\frac{x \log^2 (qxZ)}{Z} \Bigr).
\end{equation}
\end{lem}

Since clearly $L(1/2 + i\gamma, \chi) = 0$ if and only if 
$L(1/2 - i\gamma, \overline{\chi}) = 0$, we have
\begin{align}
\notag
\sum_{\substack{-Z \leq  \gamma \leq  0 \\L(1/2 + i\gamma, \chi) = 0
}} \frac{x^{1/2 + i\gamma}}{1/2 + i\gamma}
&=
\sum_{\substack{0 \leq  \delta \leq  Z \\L(1/2 - i \delta, \chi) = 0
}} \frac{x^{1/2 - i\delta}}{1/2 - i\delta}
=
\sum_{\substack{0 \leq  \delta \leq  Z \\L(1/2 + i\delta, \overline{\chi}) = 0
}} \frac{x^{1/2 - i\delta}}{1/2 - i\delta}
\\&
\label{conj-trick}
=
\sum_{\substack{0 \leq  \delta \leq  Z \\L(1/2 + i\delta, \overline{\chi}) = 0
}} 
\overline{\Bigl(\frac{x^{1/2 + i\delta}}{1/2 + i\delta}\Bigr)}.
\end{align}

This allows us to rewrite \eqref{explform-all-gammas} as 
\begin{equation}
\label{explform-positive-gammas}
\psi(x,\chi) = -\!\!\!\sum_{\substack{0 < \gamma \leq  Z \\ L(1/2 + i\gamma, \chi) = 0
}} \frac{x^{1/2 + i\gamma}}{1/2 + i\gamma} 
-\!\!\!
\sum_{\substack{0 \leq  \gamma \leq  Z \\L(1/2 + i\gamma, \overline{\chi}) = 0
}} 
\overline{\Bigl(\frac{x^{1/2 + i\gamma}}{1/2 + i\gamma}\Bigr)}
+
O\Bigl(\frac{x \log^2 (qxZ)}{Z} \Bigr),
\end{equation}
which involves only zeros having non-negative imaginary parts
belonging to two different $L$-functions attached to non-principal characters.
Identity \eqref{conj-trick} will be useful in the proofs of Lemmas \ref{conj-trick-lemma}-\ref{lem21}.

We will apply the Cauchy--Schwarz inequality ($L^2$-norm form) multiple times 
in our proofs so we record it here: 
for any square-integrable complex valued functions $f$ and $g$, we have
\begin{equation}
\label{Cauchy}
\Bigl \vert  \int_a^b f(t) g(t) \, \dx t \Bigr \vert  
\leq 
\Bigl( \int_a^b  \vert f(t) \vert ^2  \, \dx t \Bigr)^{1/2} 
\Bigl( \int_a^b  \vert g(t) \vert ^2 \, \dx t \Bigr)^{1/2}.
\end{equation}

We also record the Sobolev--Gallagher inequality (which will play an important role in proving Lemma \ref{lem3}).
\begin{lem} \cite[Lemma~1.1]{MontgomeryTopics}
\label{Sob-Gal-ineq}
Let $a<b$ be real numbers and $f$ a continuous complex valued function on
$[a,b]$, with continuous first derivative $f'$ in $(a,b)$. Then
\begin{equation*}
\vert f(u) \vert  \leq (b-a)^{-1} \int_a^b  \vert f(t) \vert  \, \dx t + \int_a^b  \vert f'(t) \vert  \, \dx t
\end{equation*}
for any $u$ in $[a, b]$.
\end{lem}

Finally, we state
the following crucial lemma, an analogue of Lemma 6 of Goldston--Montgomery
\cite{GM1984}. In fact, the statement of Lemma \ref{gmslem} is different from the one in \cite[Lemma~6]{GM1984} 
since  we are integrating over $[-T,T]$, rather
than over $[0,T]$. This causes 
the presence of  the extra factor $2$ in the main term (remark
that the coefficients $c(\mu)$ are real numbers and hence 
$ \vert \mathcal{S}(t) \vert^2 =  \vert \mathcal{S}(-t) \vert ^2$ for $t\in [0,T]$).
\begin{lem}
\label{gmslem}
Let $\mathcal{S}(t) \coloneq \sum_{\mu \in \mathcal{M}} c(\mu) e^{2\pi i\mu t}$
be a Fourier series with $\mathcal{M}$ be a countable set of real numbers and with 
$c(\mu)$ real Fourier coefficients. 
If $\sum_{\mu \in \mathcal{M}}  \vert c(\mu) \vert  < \infty$,
then uniformly for $T\ge 1$ and $1/2T\leq \delta \leq 1/2$, we have
\[
\int_{-T}^{T}  \vert \mathcal{S}(t) \vert ^2 \, \dx t = 
2T \sum_{\mu \in \mathcal{M}}  \vert c(\mu) \vert ^2 + 
O\Big( \delta^{-1} \sum_{\mu \in \mathcal{M}}  \vert c(\mu) \vert ^2 
+ 
T\!\! \sum_{\substack{\mu, \nu \in \mathcal{M} \\
0 <  \vert \mu-\nu \vert < \delta}}  \vert c(\mu) c(\nu) \vert \Big).
\]
\end{lem}

\section{Proofs of Theorems \ref{yalmodifiedcor} and \ref{yalmodifiedthm1}}
\label{S3}

Under GRH, following the argument in Landau \cite[p.~353]{LandauHand} we obtain for $q\ge 1$,
$(a,q)=1$, $x > 1$, and for primitive characters $\chi$ (when $q>1$) that
\[
\sideset{}{'} \sum_{n\leq  x} \frac{\Lambda(n)\chi(n)}{n^s}
=
\delta_q \frac{x^{1-s}}{1-s}
+
\sum_{\ell=0}^{\infty}\frac{x^{-2\ell-s-\mathfrak{a}}}{2\ell+s+\mathfrak{a}}
-
\sum_{\rho}\frac{x^{\rho-s}}{\rho-s}
-
\frac{L^\prime}{L}(s,\chi),
\]
where the dash over the summation indicates that only half of
the term with 
$n = x$ is to be included in the sum. As is customary,
the sum over the non-trivial zeroes has to be interpreted in 
the symmetrical 
sense as $ \lim_{Z \to \infty} \sum_{ \vert \gamma \vert  < Z}$,
$s\in \mathbb C$, $s\ne 1$, $s\ne \rho$, $s\ne -(2\ell+\mathfrak{a})$,
$\mathfrak{a} =0$ if $\chi(-1)=1$,
$\mathfrak{a} =1$ if $\chi(-1)=-1$,
and $\delta_q=1$ if $q=1$ and $\delta_q=0$ otherwise\footnote{The term 
with $\delta_q$  gives the contribution of the pole at $1$ of the Riemann zeta function,
see the proof of eq.~(18) in \cite{Montgomery1973}, and it should be inserted into
\cite[eq.~(6)]{yal} too.}. 
Letting $s=\sigma+it$ and $\rho=1/2+i\gamma$, 
and following the proof of Montgomery \cite[pp.~185--186]{Montgomery1973},
the previous formula can be rewritten as
\begin{equation}
\label{switchLR}
\sum_{\rho}
\frac{x^{1/2-\sigma}x^{i(\gamma-t)}}{\sigma-1/2 +i(t-\gamma)}
=
 \frac{L^\prime}{L}(s,\chi)
 +
\sideset{}{'} \sum_{n\leq  x} \frac{\Lambda(n)\chi(n)}{n^s}
-
\delta_q \frac{x^{1-s}}{1-s}
-
\sum_{\ell=0}^{\infty}\frac{x^{-2\ell-s-\mathfrak{a}}}{2\ell+s+\mathfrak{a}}
.
\end{equation}
On replacing $s$ by $1-\sigma+it$, we obtain
\begin{align}
\notag
\sum_{\rho}\frac{x^{\sigma-1/2}x^{i(\gamma-t)}}{1/2-\sigma +i(t-\gamma)}
&=
\frac{L^\prime}{L}(1-\sigma+it,\chi)
 +
\sideset{}{'} \sum_{n\leq  x} \frac{\Lambda(n)\chi(n)}{n^{1-\sigma+it}}
\\&
\label{reflection}
\hskip1cm
-
\delta_q \frac{x^{\sigma-it}}{\sigma-it}
-
\sum_{\ell=0}^{\infty}\frac{x^{-2\ell-1 + \sigma -it-\mathfrak{a}}}{2\ell+1 - \sigma +it+\mathfrak{a}}
.
\end{align}
Subtracting the respective sides of \eqref{reflection} from \eqref{switchLR} 
and using the relation 
\begin{equation}
\label{derlog-L}
 \frac{L^\prime}{L}(s,\chi)
 =
 - \sum_{n\ge 1} \frac{\Lambda(n)\chi(n)}{n^s}
 \quad (\Re(s)>1),
\end{equation}
 we obtain 
\begin{align}
\notag
\sum_{\gamma}\!
\frac{(2\sigma-1) x^{i\gamma}}{(\sigma -1/2)^{2}+(t-\gamma)^{2}}
&=
-\frac{1}{\sqrt{x}} 
\Bigl(
\sideset{}{'} \sum_{n\leq x} \Lambda(n)\chi(n) \Bigl( \frac{x}{n} \Bigr)^{1-\sigma+it} 
\!+\!
\sum_{n > x} \Lambda(n)\chi(n) \Bigl( \frac{x}{n} \Bigr)^{\sigma+it} 
\Bigr)
\\
\notag
&
-
\frac{L^\prime}{L}(1-\sigma+it, \chi)
x^{1/2-\sigma+it} 
+
\delta_q 
\frac{(2\sigma-1)x^{1/2}}{(\sigma-1+it)(\sigma-it)}
\\
\label{Montgomery-eq-(22)}
&
-\frac{1}{\sqrt{x}} 
\sum_{\ell=0}^{\infty}
\frac{(2\sigma-1) x^{-2\ell-\mathfrak{a}}}{(\sigma-1-it-2\ell-\mathfrak{a} ) (\sigma+it+2\ell+\mathfrak{a})}.
\end{align}
For primitive $\chi \pmod{q}$,  
we obtain by logarithmic differentiation of the functional equation of $L(s,\chi)$ that
\begin{equation*}
- \frac{L^\prime}{L}(1-\sigma+it, \chi)
=
\log \Bigl(\frac{q\tau}{2\pi} \Bigr)
+ \frac{L^\prime}{L}(\sigma-it, \overline{\chi})
+O_\sigma (1),
\end{equation*}
with \( \tau =  \vert t \vert +2\), 
which,  using \eqref{derlog-L}, for $\sigma>1$ becomes
 \begin{equation}
\label{derlog-reflection}
- \frac{L^\prime}{L}(1-\sigma+it, \chi)
=
\log (q\tau)
+O_\sigma (1).
\end{equation}
Letting $\sigma=3/2$, inserting \eqref{derlog-reflection} into
\eqref{Montgomery-eq-(22)}, multiplying both sides of \eqref{Montgomery-eq-(22)} by $\overline{\chi}(a)$
and summing over the Dirichlet characters, we can write 
\begin{align}\label{eqn8}
\Bigl \vert \sum_{\chi \pmod{q}} \overline{\chi}(a) 
\sum_{\substack
{\gamma
}} \frac{2x^{i\gamma}}{1+(t-\gamma)^2} \Bigr \vert ^2
= \Bigl \vert \sum_{1\leq  j \leq  4} R_j(x,t)\Bigr \vert ^2,
\end{align}
where 
\begin{align*}
R_1(x,t) & \coloneq 
- \frac{\varphi(q)}{\sqrt{x}} 
   \Bigl( 
   \sum_{\substack{n \leq x \\  n \equiv a \pmod{q}}} \Lambda(n) \Bigl(\frac{x}{n}\Bigr)^{-1/2+it}
   + 
   \sum_{\substack{n > x \\ n \equiv a \pmod{q}}}\Lambda(n) 
   \Bigl(\frac{x}{n}\Bigr)^{3/2+ it}
   \Bigr),\\ 
 R_2(x,t) &\coloneq x^{-1+it} \sum_{\chi \pmod{q}} \overline{\chi}(a)
  ( \log (q^*\tau) +O(1)), 
\\
R_3(x,t) & \coloneq O( x^{-1/2}\tau^{-1} 
\varphi(q)  + \delta_q x^{1/2}\tau^{-2}), 
\\
R_4(x,t) & \coloneq \frac{1}{\sqrt{x}} \sum_{\chi \pmod{q}} \overline{\chi}(a) \Bigl(\!\!\!
\sum_{\substack{n \leq x \\ \chi(n) \ne \chi^*(n)}
} \!\!\! \Lambda(n) \chi^*(n) \Bigl(\frac{x}{n}\Bigr)^{-1/2 +it}
\\&
\hskip6cm+\!\!\!\!\!
\sum_{\substack{n > x \\ \chi(n) \ne \chi^*(n)}
} \!\!\!\Lambda(n) \chi^*(n) \Bigl(\frac{x}{n}\Bigr)^{3/2+ it}
\Bigr).
\end{align*} 

Note that we have used the orthogonality of Dirichlet characters to obtain the expression of $R_1$.
The precise definition of $R_2$ and $R_3$ is not relevant for our purposes, the estimates above will suffice. 
In the definition of $R_4(x,T)$, \(\chi \pmod{q}\) is a character induced by the primitive character \(\chi^*\pmod{ q^*}\). 
The last term in the parentheses is a correction term for non 
primitive \(\chi\).\footnote{\label{foot_typo}Remark that in the 
analogous formula in \cite[eq.~(8)]{yal} in the
contribution for primitive characters for $x^{1/2}$ one should read $x^{-1/2}$.}
Moreover, since $\chi(n) \ne \chi^*(n)$ implies that $(n,q) > 1$,
for such $n$ we have $\sum_{\chi \pmod{q}} \overline{\chi}(a) \chi^*(n) = 0.$ 
As a consequence, we obtain that
$R_4(x,t)=0$.

\begin{proof}[Proof of Theorem \ref{yalmodifiedcor}]

We integrate both sides of equation \eqref{eqn8} from $t=-T$ to $t=T$,
where $T$ will be specified later.
The left hand side of equation \eqref{eqn8} can be written as 
\begin{equation*}
    \sum_{\chi_1, \chi_2} \overline{\chi_1}(a) \chi_2(a) 
    \sum_{\substack{\gamma_j, \, j=1,2 
    }}
   \frac{4 x^{i(\gamma_1 - \gamma_2)}}{(1+(t -\gamma_1)^2) (1+(t-\gamma_2)^2)},
\end{equation*}
where in the double sum over the characters we dropped, for brevity, the $\!\!\pmod{q}$ specification.
Now we integrate 
this function from $t=-T$ to $t=T.$ We claim that
\begin{align}
\notag
    &\int_{-T}^T \sum_{\chi_1, \chi_2} \overline{\chi_1}(a) \chi_2(a) 
    \sum_{\substack{\gamma_j, \, j=1,2
    }}
   \frac{4 x^{i(\gamma_1 - \gamma_2)}}{(1+(t -\gamma_1)^2) (1+(t-\gamma_2)^2)} \, \dx t 
   \\
\notag
    &
    =\!\!\!
    \int_{ -\infty}^{\infty} \sum_{\chi_1, \chi_2} \overline{\chi_1}(a) \chi_2(a)
    \!\!\!
    \sum_{\substack{ \vert \gamma_j \vert  \leq T \\ j=1,2
     }}
   \frac{4 x^{i(\gamma_1 - \gamma_2)}}{(1+(t -\gamma_1)^2) (1+(t-\gamma_2)^2)} \, \dx t 
   + O(\varphi(q)^2 \log T\log^2 (qT)
   \\
     \label{sona}
     &=   2\pi F_q(x,T) + O(\varphi(q)^2 \log T\log^2 (qT).
   \end{align}
To prove equation \eqref{sona}, we first recall that for $\chi$ mod $q$, there are $\ll \log (qT)$ 
zeros such that $L(1/2 + i\gamma, \chi) = 0$ and $T \leq \gamma \leq T+1$, $T \geq 2$. 
This implies for $ \vert t\vert \leq T$ that
\begin{equation}
\label{eqn33}
    \sum_{\substack{ \vert \gamma \vert  >T 
    }}
   \frac{1}{1+(t - \gamma)^2 }
   \ll 
    \frac{\log (qT)}{T-t+1}.
    \end{equation}

We now recall \cite[Ch.~16, Lemma]{Davenport}, i.e., 
\begin{equation}
\label{eqn34}
    \sum_{\substack{\gamma 
    }}
   \frac{1}{1+(t - \gamma)^2 }
   \ll \log (q \tau).
    \end{equation}
Using \eqref{eqn33}-\eqref{eqn34} we obtain  that
\begin{equation}
    \int_{ -T}^T\! \sum_{\chi_1, \chi_2} \overline{\chi_1}(a) \chi_2(a)
    \!\!\!
    \sum_{\substack{\gamma_1, \gamma_2 \\  \vert \gamma_2 \vert  >T 
    }}
    \!\!\!
   \frac{4 x^{i(\gamma_1 - \gamma_2)}}{(1+(t -\gamma_1)^2) (1+(t-\gamma_2)^2)} \, \dx t
   \label{eqn36}
   \ll 
   \varphi(q)^2 (\log T)\log^2 (qT).
\end{equation}

For $\vert t \vert > T$, we have 
\begin{align*}
\sum_{\vert\gamma \vert \leq T}
\frac{1}{1+(t-\gamma)^{2}}
&\ll
\frac{\log (q\vert t \vert)}{\vert t \vert - T + 1},
\end{align*}
so that
\begin{align}
    \int_{ \vert t \vert > T}  \sum_{\substack{ \vert \gamma_j \vert  \leq T \\ j=1,2 
    }}
   \frac{\, \dx t}{(1+(t -\gamma_1)^2) (1+(t-\gamma_2)^2)}  \nonumber 
   & \ll 
   \int_{ \vert t \vert > T}  
   \frac{\log^2 (q\vert t \vert) }{(\vert t \vert - T + 1 )^2} \, \dx t
  \ll \log^2 (qT). 
\end{align}
This implies  that
\begin{equation}
    \int_{ \vert t \vert > T}\! \sum_{\chi_1, \chi_2} \overline{\chi_1}(a) \chi_2(a)  \!\!\!
    \sum_{\substack{ \vert \gamma_j \vert  \leq T \\ j=1,2 
    }}
   \frac{4 x^{i(\gamma_1 - \gamma_2)}}{(1+(t -\gamma_1)^2) (1+(t-\gamma_2)^2)} \, \dx t 
   \label{eqn35}
   \ll \varphi(q)^2 \log^2 (qT).
\end{equation}
Now \eqref{sona} follows on 
combining equations \eqref{eqn36}-\eqref{eqn35}
and applying Cauchy's 
residue theorem to evaluate the second integral in \eqref{sona}.
As $R_4(x,t)=0$, from \eqref{eqn8}-\eqref{sona} we obtain
   \begin{align}\label{intlhsrhs}
       \int_{-T}^T \big \vert \sum_{1\leq  j \leq  3} & R_j(x,t)\big \vert ^2 \, \dx t 
  =  2\pi F_q(x,T) + O(\varphi(q)^2 \log T\log^2 (qT).
   \end{align}
Using the Cauchy--Schwarz inequality \eqref{Cauchy}, we can easily deduce that
\begin{align}\label{eqn39}
  \int_{-T}^T \big \vert \sum_{1\leq  j \leq  3} & R_j(x,t)\big \vert ^2 \, \dx t 
  =   \sum_{1\leq  j\leq  3} \int_{-T}^T \big \vert  R_j(x,t)\big \vert ^2 \, \dx t   \nonumber \\
  & + O\Bigl(\sum_{1\leq  j \leq  3} \sum_{\substack{1\leq  k \leq  3\\ k\ne j}} 
  \bigl(\int_{-T}^T \big \vert  R_j(x,t)\big \vert ^2 \, \dx t\bigr)^{1/2} 
  \bigl(\int_{-T}^T \big \vert  R_k(x,t)\big \vert ^2 \, \dx t\bigr)^{1/2}
  \Bigr).
\end{align}

Integrating   $\vert R_j(x,t) \vert^2, j = 2,3$ (see \eqref{eqn8} for 
their definitions) from $t=-T$ to $t=T$, we 
obtain\footnote{Equation \eqref{eqn411} is the analogue of the first of the two integral 
estimates appearing in the middle of \cite[p.~330]{yal}, 
where for $x$ in both argument and bound one should 
read $x^2$. Taking this into account we arrive at a better upper bound for $T$ in 
Theorem \ref{yalmodifiedthm1}.}

\begin{align}\label{eqn411}
    \int_{-T}^T  \vert R_2(x,t) \vert ^2 \, \dx t \ll
    \frac{\varphi(q)^2 T \log^2 (qT)}{x^2},
\end{align}
and 
\begin{align}\label{eqn40}
   \int_{-T}^T  \vert R_3(x,t) \vert ^2 \, \dx t \ll \frac{\varphi(q)^2}{x} +x \ll x
\end{align}
for $1\leq  q \leq  x$.

The mean square of  $R_1(x,t)$ is evaluated in the following lemma.
\begin{lem}
\label{GM-lemma}
Uniformly for $q \leq x^{1-\eps}$  and $ x/\varphi(q) \leq T$ we have, as $x \to \infty$,
\begin{align}
\label{sx11}
  \frac{1}{\varphi(q)^2} \int_{-T}^T  \vert R_1(x,t) \vert ^2 \, \dx t  &=
  2T S(x) + O\Big(
   \frac{\sqrt{TxS(x)}}{\varphi(q)}\Big),
   \end{align}
where
\[
 S(x) \coloneq \frac{1}{x^2}
\sum_{\substack{n \leq x \\
n \equiv a \pmod{q}}} n \Lambda(n)^2 +
x^2 \sum_{\substack{n > x \\ \nonumber
n \equiv a \pmod{q}}}\frac{\Lambda(n)^2}{n^3}.
\]
\end{lem}

\begin{proof}
Recalling the definition of $R_1(x,t)$ given in \eqref{eqn8}, we can write
\begin{equation}
\label{eqn388ll}
  - \frac{R_1(x,t)}{\varphi(q)} = \sum_{\substack{n \geq 1 \\  n \equiv a \pmod{q}}} c(n) \Bigl(\frac{x}{n}\Bigr)^{it}, 
   \end{equation}
where 
$c(n) \coloneq	\Lambda(n)n^{1/2}x^{-1}$, if  $n\leq x$
and $c(n) \coloneq \Lambda(n) n^{-3/2} x$, if $ n> x$.
Using \eqref{pnt}-\eqref{equi} it is clear that $\sum_{n} c(n)^2 < \infty.$
Applying Lemma \ref{gmslem} to the series in \eqref{eqn388ll}, 
we obtain that
\begin{align}\label{r1ts}
    \frac{1}{\varphi(q)^2}  \int_{-T}^T  \vert R_1(x,t) \vert ^2 \, \dx t = 2T S(x) +
  O(\delta^{-1} S(x) + E_1 + E_2),
\end{align}
where
\[
E_1 \coloneq T\!\!\! \sum_{ \substack{n < x\\
n \equiv a \pmod{q}}}\!\!\!\! \Big( \!\!\!\!\!\!\sum_{\substack{ m \leq x \\ m \equiv a \pmod{q}
\\ 0 <  \vert \log(n/m) \vert <2\pi \delta}}\!\!\!\!\! \Lambda(n) \Lambda(m) \frac{(nm)^{1/2}}{x^2}+
\!\!\!
\sum_{\substack{ m > x \\
m \equiv a \pmod{q} \\ 0 <  \vert \log(n/m) \vert <2\pi \delta}}\!\!\!\!\! \Lambda(n) \Lambda(m) \frac{n^{1/2}}{m^{3/2}}
\Big),
\]
and
\[
E_2 \coloneq T
\sum_{n \geq x}
\Big(\!\!\!\! \sum_{\substack{ m \leq x \\
m \equiv a \pmod{q} \\ 0 <  \vert \log(n/m) \vert <2\pi \delta}} 
\!\!\!\!\!\!\!\!\!\!\Lambda(n) \Lambda(m) \frac{m^{1/2}}{n^{3/2}} 
+
\!\!\!\!\!\!\!\!\!\!
\sum_{\substack{m > x \\ 
m \equiv a \pmod{q} \\
0 <  \vert \log(n/m) \vert <2\pi \delta}} 
\!\!\!\!\!\!\!\!\!\Lambda(n) \Lambda(m) \frac{x^2}{(nm)^{3/2}}
\Big).
\]

Using the Brun--Titchmarsh Classical Theorem \ref{BT-thm} 
we obtain
\begin{equation}
\label{pntsx}
   S(x) \ll \frac{\log x}{\varphi(q)},
\end{equation}
as $x \to \infty,$
uniformly in $q \leq x^{1-\eps}$.
Choosing 
\begin{equation}
\label{delta-def}
\delta = \frac{\varphi(q)}{2} \sqrt{\frac{S(x)}{Tx}},
\end{equation}
for sufficiently large $x$ we have that $1/T \leq 2 \delta \leq 1,$
when $ x/\varphi(q) \leq T$ and $ q \leq x^{1-\eps}.$
Since $ e^{2\pi \delta} - 1 \leq  10^4 \delta$ for any $\delta \leq 1/2$, we have uniformly for $q \leq x^{1-\eps}$ that

\begin{align}\label{E1est}
\nonumber
   E_1 &\ll
   \frac{T}{x} \sum_{ \substack{n \leq  x\\
n \equiv a \pmod{q}}} \sum_{\substack{ m \geq 1 \\ m \equiv a \pmod{q}
\\ 0 <  \vert \log(n/m) \vert <2\pi \delta}}\!\!\!\!\! \Lambda(n) \Lambda(m) \\
    &\nonumber
    \ll 
    \frac{T}{x}
\sum_{\substack{n \leq 
x \\ n \equiv a \pmod{q}} } \Lambda(n) 
\sum_{ \substack{n < m \leq n+ 10^4 \delta x \\ m \equiv a \pmod{q}}}  \Lambda(m) 
   \\
   &\ll \frac{T}{x}
\sum_{\substack{n \leq
x \\ n \equiv a \pmod{q}} } \Lambda(n) (\psi(n + 10^4 \delta x;q,a) -  \psi(n;q,a) )
   \ll  \frac{T \delta x}{\varphi(q)^2},
\end{align}
for sufficiently large $x$, where we have used the Brun--Titchmarsh Classical Theorem \ref{BT-thm} in the final step.  
Moreover, we can write
\begin{align}\label{E2est}
   E_2 &= T
\sum_{ r = 0}^{\infty} \sum_{ \substack{n= 2^r x\\ n \equiv a \pmod{q}}}^{2^{r+1}x}
\Big(\sum_{\substack{ m \leq x \\
m \equiv a \pmod{q} \\ 0 <  \vert \log(n/m) \vert <2\pi \delta}} 
\!\!\!\!\!\!\!\!\!\!\Lambda(n) \Lambda(m) \frac{m^{1/2}}{n^{3/2}} 
+
\!\!\!\!\!\!
\sum_{\substack{m > x \\ 
m \equiv a \pmod{q} \\
0 <  \vert \log(n/m) \vert <2\pi \delta}} 
\!\!\!\!\!\!\!\!\!\Lambda(n) \Lambda(m) \frac{x^2}{(nm)^{3/2}}
\Big) \nonumber\\ 
& \ll 
   \frac{T}{x} \sum_{r=0}^{\infty} 
    \frac{1}{2^{3r}} \sum_{ \substack{n= 2^r x\\ n \equiv a \pmod{q}}}^{2^{r+1}x}
\sum_{\substack{ m \geq 1 \\ m \equiv a \pmod{q}
\\ 0 <  \vert \log(n/m) \vert <2\pi \delta}} \Lambda(n) \Lambda(m)
\nonumber\\ 
& \ll 
   \frac{T}{x} \sum_{r=0}^{\infty} 
    \frac{1}{2^{3r}} 
    \sum_{\substack{n \leq  2^{r+1} x\\ n \equiv a \pmod{q}}}
\sum_{\substack{ m \geq 1 \\ m \equiv a \pmod{q}
\\ 0 <  \vert \log(n/m) \vert <2\pi \delta}} \Lambda(n) \Lambda(m)
\nonumber\\ 
& \ll  \sum_{r=0}^{\infty}  \frac{1}{2^{3r}}\Bigl( \frac{T \delta 2^{2r}x}{\varphi(q)^2}\Bigr)
   \ll \frac{T \delta x}{\varphi(q)^2},
\end{align}
as $x \to \infty,$
uniformly in $q \leq x^{1-\eps}$, where the final single sum  is obtained on using \eqref{E1est}.
Combining  \eqref{r1ts} and \eqref{E1est}-\eqref{E2est}, and substituting the value of $\delta$ 
as in \eqref{delta-def}, we have completed the proof of Lemma \ref{GM-lemma}.
\end{proof}

{}From \eqref{sx11} and \eqref{pntsx} we deduce that
\begin{align}\label{eqn3899}
\int_{-T}^T  \vert R_1(x,t) \vert ^2 \, \dx t 
  \ll \varphi(q) T \log x + (\varphi(q) Tx\log x)^{1/2},
\end{align}
as $x \to \infty$,  uniformly in the range 
$q \leq x^{1-\eps}$. 

Combining \eqref{eqn39}-\eqref{eqn40} and \eqref{eqn3899} it follows that

\begin{align}\label{eqn42new}
  \int_{-T}^T \big \vert \sum_{1\leq  j \leq  3} R_j(x,t)\big \vert ^2 \, \dx t & \ll
  \varphi(q)T \log x + (\varphi(q) Tx\log x)^{1/2} +
x +
  \frac{\varphi(q)^2 T \log^2 (qT)}{x^2} \nonumber\\
  &\ll \varphi(q) T \log x + (\varphi(q)Tx\log x)^{1/2},
 \end{align}

\noindent
uniformly in the range 
$q \leq x^{1-\eps}$ and $x/\varphi(q)  \leq T \leq \exp(\sqrt{x})$ as $x \to \infty.$
Recalling \eqref{intlhsrhs} we also need that
   $ \varphi(q) \log T\log^2 (qT) = o(T \log x),$
   as $x \to \infty$, and this implies $T\ge x$.
Summarizing, \eqref{intlhsrhs} and \eqref{eqn42new} imply that
\begin{align}\nonumber
   F_q(x,T) \ll \varphi(q) T  \log x,
\end{align}
uniformly in the range 
$q \leq x^{1-\eps}$ and $x \leq T \leq \exp(\sqrt{x})$ 
as $x \to \infty.$
\end{proof}

\begin{Rem} 
By following the proof strategy for the case of $F_q(x,T)$ mentioned above, the reader can also derive the analogous result for  $F^{+}_q(x,T)$; in this case 
it suffices to integrate from $0$ to $T$ instead of from $-T$ to $T$ and 
to replace \eqref{eqn33} with the following: for $0\leq  t \leq T$ one has that
\begin{equation*}
    \sum_{\substack{
    \gamma, \gamma\not\in[0,T]
    }}
   \frac{1}{1+(t - \gamma)^2 }
   \ll 
   \Bigl(
     \frac{1}{t+1}
     +
     \frac{1}{T-t+1}
    \Bigr)
    \log (qT).
    \end{equation*} 
We omit the details, since they mirror
ones that can be found 
in \cite{Montgomery1973}.
\end{Rem}

\begin{proof}[Proof of Theorem \ref{yalmodifiedthm1}]

Using the Prime Number Theorem in arithmetic progressions under GRH (see \eqref{pnterror})
in $S(x)$ (defined in Lemma \ref{GM-lemma}), we obtain
\begin{align}
\label{sx}
   S(x) = \frac{\log x}{\varphi(q)} + O \Big(\frac{\log^3 x} {\sqrt{x}}  \Big),
\end{align}
as $x \to \infty$, uniformly in $q \leq \sqrt{x}/ \log^{2+\eps} x$.
From Lemma \ref{GM-lemma}, see \eqref{sx11}, and \eqref{sx} we have
\begin{align}
\label{eqn38}
\int_{-T}^T  \vert R_1(x,t) \vert ^2 \, \dx t  &=
  2\varphi(q) T \log x + 
  O \bigl(x^{-1/2}\varphi(q)^2 T \log^3 x   \bigr)
  \nonumber
  \\
  \notag
  &\hskip2cm+
  O\Bigl( 
  \bigl(
  \varphi(q)Tx(\log x) (1+ x^{-1/2}\varphi(q) \log^2 x)
  \bigr)^{1/2}
  \Bigr)
  \\ 
  &= 2\varphi(q) T (\log x) ( 1+ o(1)) +
    O\bigl((\varphi(q) Tx\log x)^{1/2}\bigr),
\end{align}
as $x \to \infty$, uniformly in $q \leq \sqrt{x}/\log^{2+\eps} x.$ 
{}From  \eqref{eqn39}-\eqref{eqn40} and \eqref{eqn38} we obtain
\begin{align}
\notag
  \int_{-T}^T \big \vert \sum_{1\leq  j \leq  3} R_j(x,t)\big \vert ^2 \, \dx t & =
 2 \varphi(q)T (\log x) ( 1+ o(1)) 
 \\& \notag\hskip1cm
 + O\Bigl((\varphi(q) Tx\log x)^{1/2} +
x +x^{-2}\varphi(q)^2 T \log^2 (qT)\Bigr)
\\
  \label{eqn42}
  &= 2\varphi(q) T (\log x) (1+o(1)),
 \end{align}
 
\noindent 
as $x \to \infty,$
uniformly in $q \leq \sqrt{x}/\log^{2+\eps} x$,
when $x/\varphi(q) \leq T \leq \exp(x^{3/4})$.
In this $q$ and $T$ range, we also have
   $ \varphi(q) \log T\log^2 (qT) = o(T \log x),$
   as $x \to \infty.$
Equations 
\eqref{intlhsrhs} and \eqref{eqn42} imply that
\begin{align}\nonumber
   F_q(x,T) \sim \frac{\varphi(q)}{\pi} T \log x,
\end{align}
as $x \to \infty,$
uniformly in $q \leq \sqrt{x}/ \log^{2+\eps} x$,
when $x/\varphi(q) \leq T \leq \exp(x^{3/4})$.
\end{proof}

\section{Proofs of Theorems \ref{thmfullrange} and \ref{thm-Lzero-onehalf}
}
\label{S4}
Let $(a,q)= 1$. 
It is a simple consequence of the orthogonality of Dirichlet characters that
\[
\psi(x;q,a) = \frac{1}{\varphi(q)} \sum_{ \chi \pmod{q}} \overline{\chi}(a) \psi(x,\chi).
\]
Lemmas \ref{davenlem2}-\ref{davenlem} and \eqref{sip} imply that under GRH, and for $Z\leq  x$ that
\begin{align}\label{eqn1}
   \psi(x;q,a) 
   &= \frac{\psi(x, \chi_0)}{\varphi(q)}
   + \frac{1}{\varphi(q)} \sum_{\substack{ \chi \ne \chi_0\pmod{q}}} 
   \overline{\chi}(a) \psi(x,\chi) 
   \nonumber
    \\
   &=  \frac{\psi(x) + O(\log q \, \log x)}{\varphi(q)}
   + 
   \frac{1}{\varphi(q)} \sum_{\substack{ \chi \ne \chi_0\pmod{q}}} 
   \overline{\chi}(a) \psi(x, \chi) 
   \nonumber
    \\
   & = \frac{1}{\varphi(q)} 
   \Bigl( x -
\sum_{\substack{  \chi 
\pmod{q}}} \overline{\chi}(a) \sum_{\substack{ \vert \gamma \vert  \leq  Z 
}} \frac{x^{1/2 + i\gamma}}{1/2 + i\gamma} \Bigr) 
+ 
O\Bigl(
\frac{x \log^2 (qxZ)}{Z}
\Bigr),
   \end{align}
as $x \to \infty.$ 
Note that if $Z = x$ and $q \leq x^{1-\eps}$, then
the latter error term becomes $O(\log^2 x)$.
 
In the following, in order to handle the high-lying zeros,  
we will need a lemma relating $\Sigma^+(x,T,0)$,
defined in \eqref{Sigma-plus-def},
to $F^+_q(x,T)$, defined in \eqref{Fq-plus-def}.
\begin{lem}\label{lem3}
For $x\ge 2$ and $T > U \ge 0$, we have 
\[
\vert
\Sigma^+(x,T,0) - \Sigma^+(x,U,0)
\vert
\ll 
\bigl(T \max_{U \leq  t \leq T} F^+_q(x,t) \bigr)^{1/2}.
\]
\end{lem}

\begin{proof}
Using \eqref{Sigma-plus-def}, we define
\begin{equation}
\label{defgv}
h(x, T, U, v) = 
\Sigma^+(x,T,v) -
\Sigma^+(x,U,v) 
=
\sum_{\chi \pmod{q}}\overline{\chi}(a) \sum_{\substack{U <  \gamma \leq  T
}} x^{i\gamma}e^{i\gamma v}.
\end{equation}
Recalling \eqref{Fq-plus-def}, we note that

\begin{align}
\nonumber
   \int_{-\infty}^{\infty}
   & \vert h(x, T, U, v) \vert ^2e^{-2 \vert v \vert } \, \dx v 
  \\
   \nonumber
   &
   =
   \sum_{\substack{\chi_1, \chi_2 \pmod{q}}} \overline{\chi}_1(a)\chi_2(a) 
   \sum_{\substack{U <  \gamma_j \leq  T\\ j = 1,2}} 
   x^{i(\gamma_1-\gamma_2)}
   \int_{-\infty}^{\infty}
   e^{i(\gamma_1-\gamma_2)v}e^{-2 \vert v \vert } \, \dx v 
  \\
   \nonumber
   & =  \sum_{\substack{\chi_1, \chi_2 \pmod{q}}} \overline{\chi}_1(a)\chi_2(a) 
    \sum_{\substack{U <  \gamma_j \leq  T\\
   j = 1,2
   }} 
   x^{i(\gamma_1-\gamma_2)}
   W(\gamma_1 - \gamma_2)
   \\
   & = 
   \sum_{\substack{\chi_1, \chi_2 \pmod{q}}} \overline{\chi}_1(a)\chi_2(a)
   ~ \bigl( G^+_{\chi_1, \chi_2}(x,T) - G^+_{\chi_1, \chi_2}(x,U) \bigr)
   \nonumber\\
       \label{G-F-relation}
   & = F^+_q(x,T) - F^+_q(x,U).
\end{align}
Letting $H(v) =  \vert  h(x, T, U, v) \vert ^2$,
we first remark that, using \eqref{G-F-relation}, we have
\begin{equation}
\label{G-estim}
\int_{-1}^1 
\vert  H(v)  \vert 
\, \dx v
\ll
\int_{-\infty}^{\infty}
\vert  H(v)  \vert 
\,e^{-2 \vert v \vert }\, \dx v
\ll
F^+_q(x,T) + F^+_q(x,U).
\end{equation}
Combining the Sobolev--Gallagher inequality (Lemma \ref{Sob-Gal-ineq})
and
\eqref{G-estim} yields
\begin{align}
\notag
\vert h(x, T, U, 0) \vert ^2
= 
H(0) 
&
\ll \int_{-1}^1  \vert  H(v)  \vert  \, \dx v +  \int_{-1}^1  \vert  H^{\prime}(v)  \vert  \, \dx v  
\\
&
\label{Sob-Gall-ineq}    
\ll F^+_q(x,T) +  F^+_q(x,U) + \int_{-1}^1   \vert  H^{\prime}(v)  \vert  \, \dx v.
\end{align}
By the Cauchy--Schwarz inequality \eqref{Cauchy}, 
we obtain 
\begin{equation}
\!\!
\label{G-prime-first}
\int_{-1}^1   \vert  H^{\prime}(v)  \vert  \, \dx v  \ll  
\Bigl(\int_{-1}^1 
\vert  H(v)  \vert 
\, \dx v
\Bigr)^{1/2} 
\Bigl(\int_{-1}^1 \bigl \vert \sum_{\chi\pmod{q}}\!\!
\overline{\chi}(a)
\sum_{U <  \gamma \leq  T} \!\!
\gamma x^{i\gamma}e^{iv\gamma}\bigr \vert ^2\, \dx v
\Bigr)^{1/2}.
\end{equation}
Moreover, by partial summation we have
\[
\sum_{\chi\pmod{q}}
\overline{\chi}(a)
\sum_{U <  \gamma \leq  T}\gamma x^{i\gamma}e^{iv\gamma}
=
T h(x, T, U, v)
- 
\int_{U}^T h(x, t, U, v)\, \dx t
\]
and, again using \eqref{G-F-relation}, we obtain
\begin{align}
\notag
\int_{-1}^{1}
\bigl \vert \sum_{\chi\pmod{q}}
&
\overline{\chi}(a) 
\sum_{U <  \gamma \leq  T}\gamma x^{i\gamma}e^{iv\gamma}\bigr \vert ^2 \, \dx v
\\
\notag
&\ll 
\int_{-\infty}^{\infty}
\bigl \vert \sum_{\chi\pmod{q}}
\overline{\chi}(a) 
\sum_{U <  \gamma \leq  T}\gamma x^{i\gamma}e^{iv\gamma}\bigr \vert ^2\,e^{-2 \vert v \vert }\, \dx v
\\
\notag
&\ll 
T^2
\int_{-\infty}^{\infty} H(v)   \,e^{-2 \vert v \vert }\, \dx v
+ 
T  \int_{U}^T 
 \int_{-\infty}^{\infty} \vert h(x, t, U, v) \vert ^2  \,e^{-2 \vert v \vert }\, \dx v \,\, \dx t
\\
\notag
& 
\ll
T^2  (F^+_q(x,T) + F^+_q(x,U))
+
T  \int_{U}^T  (F^+_q(x,t) + F^+_q(x,U)) \,\, \dx t
 \\
\label{G-prime-second}  
&
\ll T^2  \max_{U\leq t\leq T}F^+_q(x,t).
\end{align}
Combining  \eqref{G-estim} and \eqref{G-prime-first}-\eqref{G-prime-second}
we obtain the estimate
\begin{equation}
\label{thesis}
\int_{-1}^1  \vert H^{\prime}(v) \vert \, \dx v  \ll T\max_{U\leq t\leq T}F^+_q(x,t)
\end{equation}
and the lemma  follows from \eqref{defgv}, \eqref{Sob-Gall-ineq} and \eqref{thesis}.
\end{proof}

For the low-lying zeros, we can link the mean-square average of $\Sigma^+(x,T,0)$ and 
$\Sigma(x,T,0)$, defined in \eqref{Sigma-def}, to $F^+_q(x,T)$, respectively $F_q(x,T)$,
defined in \eqref{Fq-def}.
This lemma will play a crucial role in proving both Theorems \ref{thmfullrange} 
and \ref{thm-Lzero-onehalf}.
\begin{lem}\label{lemlp}
For $x\ge 2$ and $T\ge 0$, we have 
\begin{equation*}
\int_{x}^{2x} \vert\Sigma(t,T,0) \vert^2 \, \dx t 
\ll 
 x F_q(x, T)
\quad 
\textrm{and}
\quad
\int_{x}^{2x} \vert \Sigma^+(t,T,0)  \vert^2 \, \dx t 
\ll 
 x F^+_q(x, T).  
\end{equation*}
\end{lem}
\begin{proof}
We only prove the first estimate, as the second can be obtained in a completely similar way.
Substituting $t= x e^{v}$ in \eqref{Sigma-def} we obtain
\begin{align*}
\int_{x}^{2x} \vert\Sigma(t,T,0) \vert^2 \, \dx t 
&= x \int_{0}^{\log 2} \vert\Sigma(x,T,v) \vert^2 e^{v} \, \dx v \nonumber\\
&\ll 
x \int_{-\infty}^{\infty}  \vert\Sigma(x,T,v) \vert^2 e^{-2 \vert v \vert } \, \dx v
= 
x F_q(x, T),
\end{align*}
where the second equality is a 
consequence of \eqref{integral-form-of-F}.
\end{proof}

The next lemma shows how to connect the contribution of the set of zeros with
$\gamma\leq  0$ to the zero-sums in \eqref{Sigma-def} and in \eqref{eqn1}
with the one obtained using the set of zeros having $\gamma\ge0$.

\begin{lem}\label{conj-trick-lemma}
If $(a, q)= 1$, $w\in \mathbb{R}$ and $T\ge0$, then
\[
\sum_{\chi \pmod{q}}\overline{\chi}(a) \sum_{\substack{-T \leq  \gamma \leq  0
}} \frac{w^{i\gamma}}{1/2 + i\gamma} 
=
\overline{
    \sum_{\chi \pmod{q}}
  \overline{\chi}(a) \!\!\! 
   \sum_{\substack{0\leq  \gamma \leq  T }} 
\frac{w^{i\gamma}}{1/2 + i\gamma}}.
\] 
\end{lem}
\begin{proof} 
Using \eqref{conj-trick} and 
the equality of the set of Dirichlet characters and the set of their conjugates, 
we have for every real number $w$,

\begin{align*}
\sum_{\chi \pmod{q}}\overline{\chi}(a) \sum_{\substack{-T \leq  \gamma \leq  0
\\L(1/2+ i\gamma, \chi) = 0
}} \frac{w^{i\gamma}}{1/2 + i\gamma} 
&=
\overline{
    \sum_{\chi \pmod{q}}
  \chi(a) \!\!\! 
   \sum_{\substack{0\leq  \gamma \leq  T \\L(1/2 + i\gamma, \overline{\chi}) = 0
}} 
\frac{w^{i\gamma}}{1/2 + i\gamma}}
\\
&=
\overline{
    \sum_{\overline{\psi} \pmod{q}}
  \overline{\psi}(a) \!\!\! 
   \sum_{\substack{0\leq  \gamma \leq  T \\L(1/2+ i\gamma, \psi) = 0
}} 
\frac{w^{i\gamma}}{1/2 + i\gamma}}
\\
&
=
\overline{
    \sum_{\psi \pmod{q}}
  \overline{\psi}(a) \!\!\! 
   \sum_{\substack{0\leq  \gamma \leq  T \\ L(1/2+ i\gamma, \psi)  = 0
}} 
\frac{w^{i\gamma}}{1/2 + i\gamma}}.
\end{align*}
The lemma follows by renaming $\psi$ as $\chi$ in the last quantity.
\end{proof}

We are now ready to  prove  Theorem \ref{thm-Lzero-onehalf}.
\begin{proof}[Proof of Theorem \ref{thm-Lzero-onehalf}]
Let $q$ be fixed and large and let $a=1$.
Given $\eps>0$, let $x$ be such that
$x^{1-2\eps} \leq  q \leq  x^{1-\eps}$. 
Put $U=x^{\epsilon}$ and set $c_0=2(\sqrt{2}-1)$.
For every $x>0$ and $\vert\gamma \vert \leq  U$ we define
\[
w(x,\gamma) 
\coloneq 
\frac{(2^{1/2+i\gamma}-1)x^{i\gamma}}{c_0\cdot(1/2+i\gamma)} 
= 
\frac{1}{c_0\sqrt{x}}
\int_{x}^{2x} t^{-1/2+i\gamma}\, \dx t.
\]
We will use $w(x,\gamma)$
to weight the contribution of the low-lying zeros on the critical line.
The value of $c_0$ was chosen so that $w(x,0) = 1$ for every positive $x$.
We now average $w(x,\gamma)$ over the zeros with imaginary parts in  $[-U,U]$
and over the characters modulo $q$.
Clearly
\[
\Sigma_1 \coloneq 
\sum_{\chi\pmod{q}} \sum_{\substack{ \vert \gamma \vert \leq U
}}  
w(x,\gamma)
=
\frac{1}{c_0\sqrt{x}}
\int_{x}^{2x}
\Bigl(
\sum_{\chi\pmod{q}} \sum_{\substack{ \vert \gamma \vert \leq U
}}  t^{i\gamma}
\Bigr)\, t^{-1/2}\, \dx t
\]
and hence, recalling \eqref{Sigma-def}, we have 
\begin{align}
\notag
\vert \Sigma_1 \vert
&=
\frac{1}{c_0\sqrt{x}}
\Bigl\vert
\int_{x}^{2x}  t^{-1/2}\, \Sigma(t,U,0;q,1)\,\dx t
\Bigl\vert
\\ \notag &
\ll
\frac{1}{\sqrt{x}}
\Bigl(\int_{x}^{2x}  t^{-1}\dx t\Bigr)^{1/2} 
\Bigl(\int_{x}^{2x} \vert  \Sigma(t,U,0;q,1) \vert^2\,\dx t\Bigr)^{1/2}
\\
\label{sigma1-estim}
&
\ll 
(F_q(x,U;1))^{1/2},
\end{align}
in which we used the Cauchy--Schwarz inequality \eqref{Cauchy} and the first part of Lemma \ref{lemlp}.
Arguing in a similar way, but now using the second part of Lemma \ref{lemlp}, we obtain
\begin{align}
\vert \Sigma_2 \vert
\coloneq 
\Bigl\vert
\sum_{\chi\pmod{q}}\sum_{\substack{0\le\gamma \leq U
}}  w(x,\gamma)
\Bigl\vert
\label{sigma2-estim}
&
\ll 
(F^+_q(x,U;1))^{1/2}.
\end{align}

Moreover, from Lemma \ref{conj-trick-lemma}, we have
\begin{equation}
\label{conj-trick2}
\Sigma_3
\coloneq 
\sum_{\chi\pmod{q}}\sum_{\substack{-U \leq  \gamma \leq  0
}}  w(x,\gamma)
=
\overline{\sum_{\chi\pmod{q}}\sum_{\substack{0\le\gamma \leq U
}}  w(x,\gamma)}
=
\overline{\Sigma_2}.
\end{equation}

We can now show how to use $w(x,\gamma)$ to estimate the 
number of $L$-functions modulo $q$ that 
vanish at $1/2$. Recalling \eqref{zero-set-def} and 
using \eqref{sigma1-estim}-\eqref{conj-trick2} we have that 
\begin{align}
\notag
\#Z_q
= 
\sum_{\substack{\chi\pmod{q}\\\ L(\frac{1}{2}, \chi)=0}}
1
&=
\sum_{\chi\pmod{q}}
\Bigl(
-
\sum_{\substack{ \vert \gamma \vert \leq U
}}  w(x,\gamma)
+
\sum_{\substack{0\leq  \gamma \leq  U
}}  w(x,\gamma)
+
\sum_{\substack{-U \leq  \gamma \leq  0
}}  w(x,\gamma)
\Bigr)
\\&
\notag
\ll
\vert \Sigma_1 \vert
+
\vert \Sigma_2 \vert
+
\vert \Sigma_3 \vert
\ll
(F_q(x,U;1) \bigr)^{1/2} + (F^+_q(x,U;1) \bigr)^{1/2}
\\
\label{sigmas-def}
&\ll
\sqrt{q U x^{\eps}}
\ll
\sqrt{q}x^{\eps}
\end{align}
in which we have used Conjecture \ref{yalconj1} for $a=1$ and $U=x^\eps$.

The proof now follows from \eqref{sigmas-def} on
recalling that $x^{1-2\eps} \leq  q \leq  x^{1-\eps}$.
\end{proof}

Our final lemma will play a crucial role in  proving Theorems \ref{thmfullrange}.
\begin{lem}\label{lem21}
Let $\eps >0$, $(a, q)= 1$ and put
\[
I(x,q,T) 
\coloneq 
\frac{1}{\varphi(q)} 
\sum_{\substack{\chi \pmod{q}}} \overline{\chi}(a) \sum_{\substack{ \vert \gamma \vert  \leq  T 
}} \frac{x^{1/2+ i\gamma}}{1/2 + i\gamma}.
\]
Let $q\ge x^\eps$, $x>q^{1+2\eps}$ and $J$ be the maximal integer such that $(x/2^J)^{1-\eps}\ge q$.
Assuming GRH and Conjecture \ref{yalconj1} in the ranges $x^\eps < q \leq  x^{1-\eps}$ and
$x^{\eps} \leq  T < x$,
we have for $j=0,\dots,J$,
\[
\Bigl \vert  I\bigl(\frac{x}{2^j},q,T\bigr)- I\bigl(\frac{x}{2^{j+1}},q,T\bigr)\Bigr \vert  \ll \sqrt{\frac{x}{2^jq}}\,x^{\eps},
\]
uniformly for
$x^\eps < q  \leq (x/2^j)^{1-\eps}$ 
and $x^{\eps} \leq T \leq  x/2^j$.
\end{lem}

\begin{proof} 
We first focus on the contribution of the non-negative 
imaginary parts in the zero-sum embedded into $I(x,q,T)$.

Let $U = x^{\eps}$.
Let $y=x/2^{j+1}$, $j=0,\dotsc,J$, $q\leq  y^{1-\eps}$. 
We first handle the high-lying zeros.
Using partial summation, Conjecture \ref{yalconj1} and Lemma \ref{lem3}, for $c=1,2$ we obtain
\begin{align*}
\notag
  \Big \vert \!\!\! \sum_{\chi \pmod{q}}
  \overline{\chi}(a) \!\!\! \sum_{\substack{U <  \gamma \leq  T}
  } &\frac{(cy)^{ i\gamma}}{1/2 + i\gamma} \Big \vert  
  \nonumber
   \ll  \frac{1}{T} \vert \Sigma^+(cy,T,0) - \Sigma^+(cy,U,0) \vert
   \\& \hskip3cm
\notag
  +\!\!
   \int_{U}^{T} 
   \vert
   \Sigma^+(cy,w,0) - \Sigma^+(cy,U,0)
   \vert
  \, w^{-2}\, \dx w \\
  \nonumber
  & \ll \frac{1}{\sqrt{T}} \Bigl(
\max_{U \leq  t \leq T}
F^+_q(cy,t) \Bigr)^{1/2} +\!\!
\int_{U}^{T} \!\! \Bigl(
\max_{U \leq  t \leq w}
F^+_q(cy,t) \Bigr)^{1/2} w^{-3/2}\, \dx w
\\
&\ll  \sqrt{q} \, y^{\eps}
\ll  \sqrt{q} \, x^{\eps},
\end{align*}
as $x \to \infty,$ uniformly for $q\leq y^{1-\eps}$.
This implies that for $U < T \leq  y$, we have
  \begin{align}\label{treat1}
\Big \vert  \sum_{\chi \pmod{q}}\overline{\chi}(a) \sum_{\substack{ U <  \gamma \leq  T}
 } \frac{(2y)^{1/2+ i\gamma}-y^{1/2+i\gamma}}{1/2 + i\gamma} \Big \vert   
\ll \sqrt{qy} x^{\eps},
   \end{align}
as $x \to \infty,$ uniformly in $q\leq y^{1-\eps}$.
Using Lemma \ref{conj-trick-lemma}, from \eqref{treat1} we also obtain
\begin{equation}
\label{negative-interval}
\Bigl\vert
\sum_{\chi \pmod{q}}\overline{\chi}(a) \sum_{-T \leq  \gamma < -U}
\frac{(2y)^{1/2+ i\gamma}-y^{1/2+i\gamma}}{1/2 + i\gamma}
\Bigr\vert
\ll \sqrt{qy} x^{\eps}.
\end{equation}

We now handle the low-lying zeros.
Recall that $y=x/2^{j+1}$, $j=0,\dotsc,J$. 
Using  Conjecture \ref{yalconj1} and Lemma \ref{lemlp}, we have
\begin{align}\label{treat12}
\Big \vert  \sum_{\chi \pmod{q}}\overline{\chi}(a)
&  \sum_{\substack{ \vert \gamma \vert \leq  U
 }} 
 \frac{(2y)^{1/2+ i\gamma}-y^{1/2+i\gamma}}{1/2 + i\gamma} \Big \vert   
 =
\Big \vert  \sum_{\chi \pmod{q}}\overline{\chi}(a) \sum_{\substack{ \vert \gamma \vert \leq  U
 }} 
 \int_{y}^{2y} w^{-1/2+i\gamma}\, \dx w \Big \vert   
 \nonumber\\
& =
 \Big \vert \int_{y}^{2y} w^{-1/2} 
 \Sigma(w,U,0)
\, \dx w\Big \vert 
 \ll
 \Big( \int_{y}^{2y} 
 \vert\Sigma(w,U,0) \vert^2
\, \dx w\Big)^{1/2}
 \nonumber\\
 &\ll 
  (y F_q(y, U))^{1/2}
 \ll (q U y^{1+\eps})^{1/2} \ll \sqrt{qy} x^{\eps},
\end{align}
as $x \to \infty$, where we have used the Cauchy--Schwarz inequality \eqref{Cauchy}
in the previous estimate.

Combining \eqref{treat1}-\eqref{treat12} and recalling $q/\varphi(q)\ll \log \log  q$,
we obtain 
\begin{equation*}
\frac{1}{\varphi(q)} 
\Bigl\vert
\sum_{\chi \pmod{q}}\overline{\chi}(a) \sum_{\vert \gamma \vert \leq  T}
\frac{(2y)^{1/2+ i\gamma}-y^{1/2+i\gamma}}{1/2 + i\gamma}
\Bigr\vert
\ll \sqrt{\frac{y}{q}} x^{\eps},
\end{equation*}
and this completes the proof.
\end{proof}

\begin{proof}[Proof of Theorem \ref{thmfullrange}]
Let $x>q^{1+2\eps}$ and $J$ be the maximal integer such that $(x/2^J)^{1-\eps}\ge q$.
Using the Brun--Titchmarsh Theorem (Classical Theorem \ref{BT-thm}), 
Lemma \ref{lem21} and equation \eqref{eqn1} with $x^{\eps} < q\leq x^{1-2\eps}$, 
$Z=x/2^j$,
$j=0,\dotsc, J$, we have
\begin{align}\notag
\psi(x;q,a)- \frac{x}{\varphi(q)}
&= \sum_{j=0}^{J} \Bigl(
\psi\bigl(\frac{x}{2^j};q,a\bigr)
- 
\psi\bigl(\frac{x}{2^{j+1}};q,a\bigr) - \frac{x}{2^{j+1} \varphi(q)}
\Bigr) 
\\&\hskip2cm
\nonumber
+
\psi\bigl(\frac{x}{2^{J+1}};q,a\bigr) - \frac{x}{2^{J+1} \varphi(q)}
\\ \nonumber
& \ll
\sum_{j=0}^{J} 
\sqrt{\frac{x}{2^j q}}  x^{\eps} +
\frac{x}{2^{J+1} \varphi(q)} 
\ll \sqrt{\frac{x}{ q}} 
x^{\eps}.
\end{align}

Theorem \ref{thmfullrange} now follows on recalling that in the remaining interval $1\leq  q\leq  x^\eps$ the GRH assumption by itself already 
implies the result.
\end{proof}
 
\medskip
\noindent \textbf{Acknowledgment}. 
The authors thank the anonymous referee for his/her time and attention to details in our paper as well as for 
his/her positive and insightful comments that greatly helped us to improve and clarify the presentation of 
our results.
The first author would like to thank the University of Hong Kong, 
Dalhousie University and the Max-Planck-Institut f\"ur Mathematik in Bonn 
for the hospitality and excellent working conditions.
The third author also expresses his gratitude to the latter institute. 
We thank C.Y.~\Yildirim for 
helpful email correspondence and W.~Heap for 
suggesting some improvements.
We  thank A.~Granville for an intense e-mail discussion
that helped us to improve our work.
In particular, he pointed out the importance of having control on how often $L(\frac{1}{2},\chi)=0$.

\medskip
\noindent \textbf{Data availability.} Not applicable.

\medskip
\noindent \textbf{Competing Interests.} The authors have no competing interests to declare that are relevant to the content of this article.


\end{document}